\documentclass[12pt]{amsart}
\usepackage{srcltx}

\newcommand{\baton}[1]{\mathbb #1}
\newcommand{\C}{{\baton C}}

\newcommand{\E}{{\baton E}}

\newcommand{\N}{{\baton N}} 

\newcommand{\R}{{\baton R}}     
\newcommand{\T}{{\baton T}}
\newcommand{\Z}{{\baton Z}}

\newcommand{\CC}{{\mathcal C}} 
\newcommand{\CD}{{\mathcal D}}

\newcommand{\CF}{{\mathcal F}}

\newcommand{\CH}{{\mathcal H}}
\newcommand{\CI}{{\mathcal I}}

\newcommand{\CP}{{\mathcal P}}    
\newcommand{\CQ}{{\mathcal Q}}  
 
\newcommand{\CS}{{\mathcal S}}

\newcommand{\bt}{{\mathbf t}}
\newcommand{\bx}{{\mathbf x}}
\newcommand{\bu}{{\mathbf u}}
\newcommand{\bj}{{\mathbf j}}

\newcommand{\bzero}{{\vec 0}}
\newcommand{\vt}{{\vec{t}}}
\newcommand{\vs}{{\vec{s}}}
\newcommand{\vu}{{\vec{u}}}

\newcommand{\valpha}{{\vec{\alpha}}}
\newcommand{\vepsilon}{{\vec{\epsilon}}}
\newcommand{\veta}{{\vec{\eta}}}
\newcommand{\vtheta}{{\vec{\theta}}}

\newcommand{\norm}[1]{\lVert #1\rVert}
\newcommand{\nnorm}[1]{\lvert\!|\!| #1|\!|\!\rvert}

\newcommand{\one}{{\boldsymbol 1}}
\newcommand{\inv}{^{-1}}
\newcommand{\dis}{\displaystyle}
\newcommand{\wh}{\widehat}
\newcommand{\wt}{\widetilde}

\DeclareMathOperator{\Proj}{P}
\DeclareMathOperator{\uap}{U\!A\!P}

\newtheorem{theorem}{Theorem}[section]
\newtheorem{lemma}[theorem]{Lemma}
\newtheorem{proposition}[theorem]{Proposition}
\newtheorem{corollary}[theorem]{Corollary}
\newtheorem{question}[theorem]{Question}

\newtheorem*{CSG}{Cauchy-Schwarz-Gowers Inequality}
\newtheorem*{inverse}{Inverse Theorem for Gowers Norms 
(Green, Tao, and Ziegler~\cite{GTZ})}
\newtheorem*{dualinverse}{Inverse Theorem, Dual Form}
\newtheorem*{dualinverse2}{Inverse Theorem, Reformulated Version}

\theoremstyle{definition}
\newtheorem{definition}[theorem]{Definition}

\theoremstyle{remark}
\newtheorem{remark}[theorem]{Remark}

\begin{document}

\title{A point of view on Gowers uniformity norms}
\author{Bernard Host}
\address{Laboratoire d'analyse et de math\'ematiques appliqu\'{e}es, 
Universit\'e Paris-Est Marne-la-Vall\'ee \& CNRS UMR 8050\\
5 Bd. Descartes, Champs sur Marne\\
77454 Marne la Vall\'ee Cedex 2, France}
\email{bernard.host@univ-mlv.fr}

\author{Bryna Kra}
\address{ Department of Mathematics, Northwestern University \\ 
2033 Sheridan Road \\Evanston, IL 60208-2730, USA} 
\email{kra@math.northwestern.edu}

\begin{abstract}

Gowers norms have been studied extensively both in the direct 
sense, starting with a function and understanding the associated norm, 
and in the inverse sense, starting with the norm and deducing properties 
of the function.  Instead of
focusing on the norms themselves, we study associated 
dual norms and dual functions.  Combining this study 
with a variant of the Szemer\'edi Regularity Lemma, we
give a decomposition theorem for dual functions, 
linking the dual norms to classical norms and 
indicating that the dual norm is easier to understand than the 
norm itself.  Using the dual functions, 
we introduce higher order algebras that are analogs of the 
classical Fourier algebra, which in turn can be used to 
further characterize the dual functions. 
\end{abstract}

\thanks{The first author was partially supported by the Institut Universitaire de France and the second author by NSF grant 0900873.}

\maketitle

\section{Introduction}

In his seminal work on Szemer\'edi's Theorem, 
Gowers~\cite{gowers} introduced uniformity norms $U(d)$ for 
each integer $d\geq 1$, now referred to as Gowers norms or Gowers uniformity norms, that have played an 
important role in the developments in additive combinatorics 
over the past ten years.  In particular, Green and Tao~\cite{GT} 
used Gowers norms as a tool in their proof that the primes contain 
arbitrarily long 
arithmetic progressions in the primes; shortly thereafter, they made a 
conjecture~\cite{GT3}, the Inverse Conjecture for the Gowers norms,  on the algebraic structures 
underlying these norms.  
Related seminorms were introduced by the 
authors~\cite{HK1} in the setting of ergodic theory,  and the ergodic 
structure theorem provided a source of motivation in the formulation of the Inverse Conjecture.  For each integer $d\geq 1$ and $\delta > 0$, 
Green and Tao introduce a class  
 $\CF(d,\delta)$ of ``$(d-1)$-step nilsequences 
of bounded complexity,'' which we do not define here, and the proof 
of the Inverse Conjecture was given:
\begin{inverse}
\label{th:inverse}
For each integer $d\geq 1$ and $\delta > 0$, there exists a constant 
$C = C(d, \delta)> 0 $ such that for every function $f$ on $\Z/N\Z$ 
with $|f|\leq 1$ and $\norm f_{U(d)}\geq \delta$, there exists $g\in\CF(d, \delta)$ 
with $\langle g;f\rangle \geq C$.
\end{inverse}
See also Szegedy's approach to the Inverse Conjecture, outlined 
in the announcement~\cite{CS} for the article~\cite{S}.

We are motivated by the work of Gowers in~\cite{gowers1}.  Several 
ideas come out of this work, in particular the motivation 
that algebra norms are easier to study.  
The Gowers norms $U(d)$ are classically defined
in $\Z/N\Z$, but we choose to work in a general compact abelian group.  
For most of the results presented here, we take care to 
distinguish between the group $\Z/N\Z$ and the interval $[1,\ldots, N]$,
of the natural numbers $\N$, whereas for applications in additive 
combinatorics, the results may be more 
directly proved without this separation. 
This is a conscious choice that 
allows us to separate what about Gowers norms is particular 
to  the combinatorics of $\Z/N\Z$ and what is more general.
Our point of view is  that of harmonic analysis, rather than combinatorial.

More generally, the Gowers norms can be defined on a nilmanifold.
This is particularly important in the ergodic setting where analogous
seminorms were defined by the authors in~\cite{HK1} in an arbitrary measure
space; these seminorms are exactly norms when the space is a
nilmanifold.  While we restrict ourselves to abelian groups in this
article, most of the results can be carried out in the more general
setting of a nilmanifold without significant changes.

Instead of focusing on the Gowers norms themselves, we study 
the associated dual norms that fit within this framework and the 
associated dual functions. Moreover, in the statement 
of the inverse theorem, and more generally in uses of the Gowers 
norms, one typically assumes that the functions are bounded by $1$.  
>From the duality point of view, instead we study functions in the dual space itself,
we can consider functions that are within a small  $L^{1}$ error from functions in this space.  This allows us to restrict ourselves 
to dual functions of functions in a certain $L^{p}$ class 
(Theorem~\ref{th:k}).
Moreover, we  rephrase the Inverse Theorem in terms of 
dual functions 
(see Section~\ref{subsec:def_dual_functions} for 
precise meanings of the term) in 
certain $L^p$ classes, and in this form the Gowers norms do not 
appear explicitly (Section~\ref{subsec:approx}).
This reformulates the Inverse Theorem more in a classical 
analysis context.

The dual functions allow us to introduce algebras of functions on the 
compact abelian group $Z$.  For $d=2$, this corresponds to the classical 
Fourier algebra.
Finding an interpretation for the higher order uniformity 
norms is hard and no analogs of Fourier analysis and simple 
formulas, such as Parseval, exist.  
For $d> 2$, the higher order Fourier algebra are analogs of 
the classical case of the Fourier algebra.  
These algebras allow us to further describe the dual 
functions.  Starting with a dual function of level $d$, we find that it lies 
in the Fourier algebra of order $d$, giving us information on its 
dual norm $U(d)^{*}$, and by an approximation result, we understand 
further the original function.

We obtain a result on compactness (Theorem~\ref{th:main}) of dual functions, 
by applying a variation of the classical Szemer\'edi Regularity Lemma.

\section{Gowers norms: definition and elementary bounds }

\subsection{Notation}
Throughout, we assume that 
$Z$ is a compact abelian group and let $\mu$ denote Haar measure on $Z$.
If $Z$ is finite, then $\mu$ is the uniform measure; the classical 
case to keep in mind is when $Z = \Z_{N}=\Z/N\Z$ and the measure 
of each element is $1/N$.  

All functions are implicitly assumed to be real valued.  When $Z$ is 
infinite, we also implicitly assume that all functions 
and sets are measurable.
For $1\leq p\leq\infty$, $\norm\cdot_p$ denotes the $L^p(\mu)$ norm; 
if there is a need to specific the measure, write 
$\norm\cdot_{L^{p}(\mu)}$
or $\norm\cdot_{L^{p}(Z)}$ when we wish to emphasize 
the space.

We fix an integer $d\geq 1$ throughout and the dependence 
on $d$ is implicit in all statements. 

We have various spaces of various dimensions: $1$, $d$, $2^{d}$.  
Ordinary letters $t$ are reserved for spaces of one dimension, 
vector notation $\vt$ for dimension $d$, and bold face 
characters $\bt$ for dimension $2^{d}$.

If $f$ is a function on $Z$ and $t\in Z$, we write $f_t$ for the 
function on $Z$ defined by
$$
 f_t(x)=f(x+t),
$$
where $x\in Z$.  
If $f$ is a $\mu$-integrable function on $Z$, we write 
$$
\E_{x\in Z}f(x)=\int f(x)\,d\mu(x)\ .
$$
We use similar notation for multiple integrals.
If $f$ and $g$ are functions on $Z$, we write
$$
 \langle f ;g\rangle =\E_{x\in Z}f(x)g(x),
$$
assuming that the integral on the right hand side is defined.

If $d$ is a positive integer, we set 
$$
 V_d=\{0,1\}^d.
$$
Elements of $V_d$ are written as 
$\vepsilon=\epsilon_1\epsilon_2\cdots\epsilon_d$, without commas or
parentheses. Writing $\vec 0 = 00\cdots 0\in V_d$, we set 
$$
 \wt V_d=V_d\setminus\{\bzero\}.
$$

For $\bx\in Z^{2^d}$, we write $\bx = (x_\vepsilon\colon \vepsilon\in
V_d)$.

 For $\vepsilon\in V_d$ and $\vt=(t_1,t_2,\dots,t_d)\in Z^d$ we 
write
$$
\vepsilon\cdot \vt=\epsilon_1t_1+\epsilon_2t_2+\cdots+\epsilon_dt_d\ .
$$

\subsection{The uniformity norms and the dual functions:
 definitions}
\label{subsec:def_dual_functions}

The \emph{uniformity norms}, or \emph{Gowers norms}, $\norm 
f_{U(d)}$, $d\geq 2$, of a function $f\in L^\infty(\mu)$ are defined 
inductively by
$$
 \norm f_{U(1)}=\bigl|E_xf(x)|
$$
and for $d\geq 2$,
$$
 \norm f_{U(d)}=\Bigl(E_t\norm{f.f_t}_{U(d-1)}^{2^{d-1}}\Bigr)^{1/2^{d}}.
$$
Note that $\norm \cdot_{U(1)}$ is not actually a norm.  (See~\cite{gowers} 
for more on these norms and~\cite{HK1} for a related 
seminorm in ergodic theory.)
If there is ambiguity as to the underlying 
group $Z$, we 
write $\norm\cdot_{U(Z,d)}$.  

These norms can also be defined by closed formulas:
\begin{equation}
\label{eq:def-norm}
 \norm f_{U(d)}^{2^d}=\E_{x\in Z,\; \vt\in Z^d}
\prod_{\vepsilon\in V_d}f(x+\vepsilon\cdot \vt).
\end{equation}

We can rewrite this formula. Let $Z_d$ be the subset of 
$Z^{2^d}$ defined by 
\begin{equation}
\label{eq:Zd}
 Z_d=\bigl\{ (x+\vepsilon\cdot \vt\colon\vepsilon\in V_d)
\colon x\in Z,\ \vt\in\Z^d\bigr\}.
\end{equation}
 This set can be viewed 
as the ``set of cubes of dimension $d$'' (see, for example, \cite{gowers} 
or~\cite{HK1}).  
It is easy to check that $Z_d$ is a closed subgroup of $Z^{2^d}$.
Let $\mu_d$ denote its 
Haar measure.  Then $Z_d$ is the image of
$Z^{d+1}=Z\times Z^d$ under the map $(x,\vt)\mapsto
(x+\vepsilon\cdot \vt\colon\vepsilon\in V_d)$. Furthermore, $\mu_d$ is the 
image of $\mu\times\mu\times\ldots\times\mu$ 
(taken $d+1$ times) under the 
same map. If $f_\vepsilon$, $\vepsilon\in V_d$, are functions in 
$L^\infty(\mu)$, then
$$
 \E_{x\in Z,\; \vt\in Z^d}\prod_{\vepsilon\in V_d}
f_\vepsilon(x+\vepsilon\cdot\vt)
=\int_{Z_d}\prod_{\vepsilon\in V_d}
f_\vepsilon(x_\vepsilon)\,d\mu_d(\bx).
$$
In particular, for $f\in L^\infty(\mu)$,
\begin{equation}
\label{eq:def-norm2}
\norm f_{U(d)}^{2^d}=\int_{Z_d}
\prod_{\vepsilon\in V_d}
f(x_\vepsilon)\,d\mu_d(\bx).
\end{equation}
Associating the coordinates of the set $V_d$ with 
the coordinates of the Euclidean cube, we have that 
the measure $\mu_d$ is invariant under permutations 
that are associated to the isometries of the Euclidean cube. 
These permutations act transitively on  $V_d$.

For $d=2$, by Parseval's identity we have that 
\begin{equation}
\label{eq:U2Parseval}
 \norm f_{U(2)}=\norm{\wh f}_{\ell^4(\wh Z)},
\end{equation}
where $\wh Z$ is the dual group of $Z$ and $\wh f$ is the Fourier 
transform of $f$.
For $d\geq 3$, no analogous simple formula is known and 
the interpretation of the Gowers uniformity norms is more 
difficult.  A deeper understanding of the higher order 
norms is, in part, motivation for the current work.

We make use of the ``Cauchy-Schwarz-Gowers Inequality'' (CSG)
used in the proof of the subadditivity of Gowers norms:
\begin{CSG}
\label{prop:CSG}
Let $f_\vepsilon$, $\vepsilon\in V_d$, be $2^d$ functions belonging 
to $L^\infty(\mu)$. Then
\begin{multline*}
\tag{CSG}
 \Bigl|\E_{x\in Z,\;\vt\in Z^d}f_\vepsilon(x+\vepsilon\cdot \vt)\Bigr|\\
 =
 \Bigl|\int_{Z_d} \prod_{\vepsilon\in V_d}
f_\vepsilon(x_\vepsilon)\,d\mu_d(\bx)\Bigr|
 \leq  \prod_{\vepsilon\in\{0,1\}^d}\norm{f_\vepsilon}_{U(d)}.
\end{multline*}
\end{CSG}

Applying the Cauchy-Schwarz-Gowers Inequality with half 
of the functions equal to $f$ and the other half equal to the constant $1$, 
we deduce that 
$$
\norm f_{U(d+1)}\geq\norm f_{U(d)}
$$ 
for every $f\in L^{\infty}(Z)$.

\begin{definition}
For $f\in L^\infty(\mu)$, define the \emph{dual function} $\CD_d f$ 
on $Z$ by
\begin{equation}
 \CD_df(x)=\E_{\vt\in Z^d}\prod_{\vepsilon\in \wt V_d}
f(x+\vepsilon\cdot \vt).
\end{equation}
\end{definition}
It follows from the definition that
\begin{equation}
\label{eq:norme_Dd}
 \norm f_{U(d)}^{2^d}=\langle\CD_d f;f\rangle.
\end{equation}

More generally, we define:
\begin{definition} 
If $f_{\vepsilon}\in L^{\infty}$ for $\epsilon\in\wt V_{d}$, we denote
\begin{equation}
\CD_{d}(f_{\vepsilon}\colon\vepsilon\in \wt V_{d})(x) = 
\E_{\vt\in Z^d}\prod_{\vepsilon\in \wt V_d}
f_\vepsilon(x+\vepsilon\cdot \vt).
\end{equation}
We call such a function the {\em cubic convolution product of 
the functions $f_{\vepsilon}$}.  
\end{definition}

There is a 
formal similarity between the cubic convolution product and 
the classic convolution product; for example, 
$$
\CD_{2}(f_{01}, f_{10}, f_{11})(x) = 
\E_{t_{1}t_{2}\in Z} f_{01}(x+t_1)f_{10}(x+t_2)f_{11}(x+t_1+t_2).
$$

\subsection{Elementary bounds}
For $\vepsilon\in V_d$ and $\alpha\in\{0,1\}$, 
we write $\vepsilon\alpha=\epsilon_1\dots\epsilon_d\alpha
\in V_{d+1}$, maintaining the convention that such elements 
are written without commas or parentheses.
Thus 
$$V_{d+1} = \{\vepsilon 0\colon\vepsilon\in V_{d}\}\cup
\{\vepsilon 1\colon\vepsilon\in V_{d}\}.
$$

The image of  $Z_{d+1}$ under each of the two natural projections on 
$Z^{2^d}$ is $Z_d$, and the image of the measure
$\mu_{d+1}$ under these projections is $\mu_d$.

\begin{lemma}
\label{lem:Ddf}
Let $f_\vepsilon$, $\vepsilon\in\wt V_d$, be $2^d-1$ functions 
in $L^\infty(\mu)$. Then for all $x\in Z$, 
\begin{equation}
\label{eq:boud_DProduct}
\Bigl|
\CD_{d}(f_{\vepsilon}\colon\vepsilon\in \wt V_{d})(x)\Bigr|
\leq
\prod_{\vepsilon\in \wt V_d}
\norm{f_\vepsilon}_{2^{d-1}}.
\end{equation}
In particular, for every $f\in L^\infty(\mu)$,
\begin{equation}
\label{eq:bouns_Dd}
\norm{\CD_d f}_\infty\leq\norm f_{2^{d-1}}^{2^d-1}.
\end{equation}
\end{lemma}

\begin{proof}
Without loss, we can assume that all functions are nonnegative. 
We proceed by induction on $d\geq 2$. 

For  nonnegative $f_{01}, f_{10}$ and $f_{11}\in L^\infty(\mu)$,
\begin{align*}
\CD_{2}(f_{01}, f_{10},f_{11})(x) & = \E_{t_1\in Z} f_{01}(x+t_1)\E_{t_2\in Z}
f_{10}(x+t_2)f_{11}(x+t_1+t_2) \\
& \leq \E_{t_1\in Z} f_{01}(x+t_1) \norm{f_{10}}_{L^2(\mu)}
\norm{f_{11}}_{L^2(\mu)}\\
& \leq \norm{f_{01}}_{L^2(\mu)}
\norm{f_{10}}_{L^2(\mu)}
\norm{f_{11}}_{L^2(\mu)}.
\end{align*}
This proves the case $d=2$.  Assume that the result holds for 
some $d\geq 2$. Let $f_\vepsilon$, $\vepsilon\in\wt V_{d+1}$,
be nonnegative and belong 
to $L^{2^d(\mu)}$.  Then
\begin{multline*}
\CD_{d+1}(f_{\vepsilon}\colon\vepsilon\in\wt V_{d+1})(x)\\
= \E_{\vs\in \wt Z^d}
\Bigl(
\prod_{\veta\in \wt V_d} f_{\veta 0}(x+\veta\cdot \vs) 
\E_{u\in Z}\prod_{\vtheta\in V_d} f_{\vtheta 1}
(x+\vtheta\cdot \vs+u) \Bigr).
\end{multline*}
But, for every $\vs\in Z^d$ and every $x\in Z$, by the H\"older Inequality,
$$
\E_{u\in Z}\prod_{\vtheta\in V_d} f_{\vtheta 1}
(x+\vtheta\cdot \vs+u) \leq\prod_{\vtheta\in V_d}
\norm{f_{\vtheta 1}}_{2^d}.
$$
On the other hand, by the induction hypothesis, for every $x\in Z$, 
$$
 \E_{\vs\in \wt Z^d}
\prod_{\veta\in \wt V_d}f_{\veta 0}(x+\veta\cdot \vs)
\leq\prod_{\veta\in \wt V_d}\norm{f_{\veta 0}}_{2^{d-1}}
\leq\prod_{\veta\in \wt V_d}\norm{f_{\veta 0}}_{2^{d}}
$$
and~\eqref{eq:boud_DProduct} holds for $d+1$.
 \end{proof}

\begin{corollary}
\label{cor:boud_2d}
 Let $f_\vepsilon$, $\vepsilon\in V_d$, be $2^d$ 
functions belonging to $L^\infty(\mu)$. Then
\begin{equation}
\label{eq:boud_2d}
\Bigl|\E_{x\in Z, \vt\in Z^d}\prod_{\vepsilon\in V_d}
f_\vepsilon(x+\vepsilon\cdot \vt)\Bigr|
\leq
 \prod_{\vepsilon\in V_d}\norm{f_\vepsilon}_{2^{d-1}}.
\end{equation}
In particular, for $f\in L^\infty(\mu)$,
\begin{equation}
\label{eq:boud_Norm}
\norm f_{U(d)}\leq\norm f_{2^{d-1}}.
\end{equation}
\end{corollary}

By the corollary, the definition~\eqref{eq:def-norm} 
of the Gowers norm $U(d)$ 
can be extended by continuity to the space $L^{2^{d-1}}(\mu)$, 
and if $f\in L^{2^{d-1}}(\mu)$, then the integrals defining
$\norm f_{U(d)}$ in Equation~\eqref{eq:def-norm} exist and~\eqref{eq:boud_Norm}
holds.   Using similar reasoning, if 
$f_\vepsilon$, $\vepsilon\in V_d$, are $2^d$ functions 
belonging to $L^{2^{d-1}}(\mu)$, then the integral on the left hand side 
of~\eqref{eq:boud_2d} exists, Inequality CSG remains valid, 
and~\eqref{eq:boud_2d} holds.  If we have $2^{d-1}$ functions 
in $L^{2^{d-1}}(\mu)$, then Inequality~\eqref{eq:boud_DProduct} 
remains valid.  Similarly, the definitions and results extend to 
$\CD_{d}f$ and to cubic convolution products 
for functions  belonging to $L^{2^{d-1}}(\mu)$.

The bounds given here (such as~\eqref{eq:boud_Norm}) can be improved and made sharp.  
In particular, one can show that 
$$
\norm f_{U(d)}\leq \norm f_{2^{d}/(d+1)}
$$
and 
$$
\norm{\CD f}_{\infty}\leq \norm f_{(2^{d}-1)/d}^{2^{d}-1}.
$$
We omit the proofs, as they are not used in the sequel.

When $Z$ is infinite, we define the \emph{uniform space of level $d$} 
to be the completion of 
$L^\infty(\mu)$ under the norm $U(d)$.  
As $d$ increases, the corresponding uniform spaces shrink.
A difficulty is that the uniform
space may contain more than just functions.  
For example, if 
$Z=\T:=\R/\Z$, the uniform 
space of level $2$ consists of the distributions $T$ on $\T$
whose Fourier transform $\wh T$ satisfies
$\sum_{n\in\Z}|\wh T(n) |^4 <+\infty$.

\begin{corollary}
\label{cor:alpha}
Let $f_\vepsilon$, $\vepsilon\in V_d$, be $2^d$ functions 
on $Z$  and let $\valpha\in V_d$.  Assume that 
$f_\valpha\in L^1(\mu)$ and $f_\vepsilon\in L^{2^{d-1}}(\mu)$ for 
$\vepsilon\neq\valpha$. Then
$$
\Bigl|
\E_{x\in Z,\;\vt\in Z^d} 
\prod_{\vepsilon\in V_d}
f_\vepsilon(x+\vepsilon\cdot \vt)\Bigr|\leq
\norm{f_\alpha}_1
\prod_{\substack{\vepsilon\in V_d \\ \vepsilon\neq\valpha}}
\norm{f_\vepsilon}_{L^{2^{d-1}}(\mu)}.
$$
\end{corollary}
\begin{proof}
The left hand side is equal to 
$$
\Bigl|
 \int_{Z_d} f_\valpha(x_\valpha)
\prod_{\substack{\vepsilon\in V_d \\ \vepsilon\neq\valpha}}
f_\vepsilon(x_\vepsilon)\, d\mu_d(\bx)\Bigr |
$$
Using the symmetries of the measure $\mu_d$, we can reduce to the 
case that $\valpha=\bzero$, and then the result follows 
immediately from Lemma~\ref{lem:Ddf}.  
\end{proof}

We note for later use:
\begin{lemma}
\label{prop:continuous}
For every $f\in L^{2^{d-1}}(\mu)$, $\CD_df(x)$ is a continuous function on 
$Z$.  

More generally, if $f_\vepsilon$, $\vepsilon\in\wt V_d$ are $2^d-1$ 
functions belonging to $L^{2^{d-1}}(\mu)$, then the cubic convolution 
product $\CD_{d}(f_{\vepsilon}\colon\vepsilon\in \wt V_{d})(x)$ is a 
continuous function on $Z$.  
\end{lemma}

\begin{proof}
By density and~\eqref{eq:boud_DProduct}, it suffices to prove the result when 
$f_\vepsilon\in L^\infty(\mu)$ for every $\vepsilon\in\wt V_d$.   
Furthermore, we can 
assume that $|f_\vepsilon| \leq 1$ for every $\vepsilon\in\wt V_d$.  
Let $g$ be the function on $Z$ defined in the statement.  For 
$x,y\in Z$, we have that 
$$
|g(x)- g(y) | \leq \sum_{\vepsilon\in\wt V_d}\norm{f_{\vepsilon,x} - 
f_{\vepsilon,y}}_1
$$ 
and the result follows.  
\end{proof}

\section{Duality}

\subsection{Anti-uniform spaces}
Consider the space $L^{2^{d-1}}(\mu)$ endowed with the norm $U(d)$.
By~\eqref{eq:boud_Norm}, the dual of this normed space can be viewed
as a subspace of $L^{2^{d-1}/(2^{d-1}-1)}(\mu)$, 
with  the duality given by the pairing 
$\langle\cdot ;\cdot\rangle$.
Following Green and Tao~\cite{GT}, we define
\begin{definition}
The \emph{anti-uniform space of level 
$d$} is defined to be the dual space of 
$L^{2^{d-1}}(\mu)$ endowed with the norm $U(d)$.  
Functions belonging to this space are called
\emph{anti-uniform functions of level $d$}. 
The norm on the anti-uniform space 
given by  duality is called the \emph{anti-uniform norm of level 
$d$} and is denoted by $\norm\cdot_{U(d)}^*$.
\end{definition}

Obviously, when $Z$ is finite, then every function on $Z$ is an 
anti-uniform function.
It follows from the definitions that 
$$\norm f_{U(d+1)}^{*}\leq \norm f_{U(d)}^{*}$$
for every $f\in L^{\infty}(Z)$, and thus as $d$ increases, 
the corresponding anti-uniform spaces increase.

More explicitly, a function $g\in L^{2^{d-1}/(2^{d-1}-1}(\mu)$ 
is an anti-uniform function of level $d$ if
$$
 \sup\bigl\{|\langle g ; f\rangle|\colon
f\in L^{2^{d-1}}(\mu),\ 
\norm f_{U(d)}\leq 1\bigr\} <+\infty
$$
and in this case, $\norm g_{U(d)}^*$ is defined to be equal to this 
supremum. 
Again, in case of ambiguity about the underlying space $Z$, 
we write $\norm\cdot_{U(Z,d)}^{*}$.
We conclude:

\begin{corollary}
\label{cor:anti}
For every anti-uniform function $g$ of level $d$,
$\norm g_{U(d)}^*\geq \norm g_{2^{d-1}/(2^{d-1}-1)}$.
\end{corollary}

For $d=2$, the anti-uniform space consists in functions $g\in 
L^2(\mu)$ with 
$\norm{\wh g}_{\ell^{4/3}(\wh Z)}$ finite, and for these functions,
\begin{equation}
\label{eq:U2dual}
\norm g_{U(2)}^*=\norm{\wh g}_{\ell^{4/3}(\wh Z)}.
\end{equation}  
>From this example, 
we see that there is no bound for the converse direction of 
Corollary~\ref{cor:anti}.

The dual spaces allow us to give an equivalent reformulation of the Inverse Theorem in terms of dual norms: 
For each integer $d\geq 1$ and each $\delta > 0$, there exists a family 
of  ``$(d-1)$-step nilsequences of bounded 
complexity,'' which we do not 
define here, such that its convex hull $\CF'(d,\delta)$ satisfies

\begin{dualinverse}
\label{th:dualinverse}
For each integer $d\geq 1$ and each $\delta > 0$, 
every function $g$ on $\Z_{N}$ with $\norm g_{U(d)}^{*}\leq 1$ 
can be written as $g = h+\psi$ with $h\in\CF'(d, \delta)$ and $\norm \psi_{1}\leq
\delta$.  
\end{dualinverse}
\begin{remark}
In this statement, there is no hypothesis on $\norm g_\infty$, and the 
function $g$ is not assumed to be bounded.
\end{remark}

\begin{proof}
We show that this statement is equivalent to the Inverse Theorem.
First assume the Inverse Theorem and let $\CF = \CF(d,\delta)$ be the class of 
nilsequences and $C=C(d,\delta)$ be as in the formulation of  the 
Inverse Theorem.  Let 
$$
K =\wt{\CF} + B_{L^{1}(\mu)}(C),
$$
where $\wt\CF$ denotes the convex hull of $\CF$ and the 
second term is the ball in $L^{1}(\mu)$ of radius $C$.  
Let $g$ be a function with $g\leq C$ on $K$.  In particular, $|g|\leq 1$ and 
$g\leq C$ on $\CF$.  By the Inverse Theorem, we have that 
$\norm g_{U(d)}< \delta$.  By the Hahn-Banach Theorem, 
$K \supset B_{U(d)^{*}} (C/\delta)$.  Thus 
$$
B_{U(d)^{*}} (1)\subset  (\delta /C) \wt\CF + 
B_{L^{1}(\mu)}(\delta).
$$
Taking $\CF'(d,\delta)$ to be 
$(\delta/C)\wt\CF$, we have the statement.  

Conversely, we assume the Dual Form and prove the Inverse Theorem.
Say that $\CF' = \CF'(d,\delta/2)$ is the convex hull of 
$\CF_{0}=\CF_{0}(d,\delta)$.  Assume that $f$ satisfies $|f|\leq 1$ and 
$\norm f_{U(d)}\geq\delta$.  Then there exists 
$g$ with $\norm g_{U(d)}^{*}\leq 1$ and $\langle g;f\rangle \geq
\delta$.  By the dual version, there exists $h\in\CF'$ 
and $\psi$ with $\norm\psi_{1}< \delta/2$ such that $g = h+\psi$.  Since 
$$
\delta\leq \langle g;f\rangle = \langle h+\psi;f\rangle = 
\langle h;f\rangle + \langle \psi;f\rangle
$$
and $\langle \psi;f\rangle\leq\delta/2$, we have that 
$\langle h;f\rangle\geq \delta/2$.  Since $h\in\CF'$, there exists 
$h'\in\CF_{0}$ with $\langle h';f\rangle> \delta/2$ and we have the 
statement.  
\end{proof}

\subsection{Dual functions and anti-uniform spaces}

\begin{lemma}
\label{lem:norm_Dprod}
Let $f_\vepsilon$, $\vepsilon\in\wt V_d$, belong to $L^{2^{d-1}}(\mu)$.  
Then 
$$\norm{\CD_{d}(f_{\vepsilon}\colon\vepsilon\in \wt V_{d})}_{U(d)}^*\leq\prod_{\vepsilon\in \wt 
V_d}\norm{f_\vepsilon}_{2^{d-1}}.$$
\end{lemma}

\begin{proof}
For every $h\in L^{2^{d-1}}(\mu)$, we have that
\begin{align*}
 \bigl|\langle h; g\rangle\bigr| =  &
\Bigl|\E_{x\in Z,\;\vt\in Z^d} h(x+\bzero\cdot \vt)
\prod_{\vepsilon\in\wt 
V_d}f_\vepsilon(x+\vepsilon\cdot \vt)\Bigr|\\
\leq & \norm h_{U(d)}\cdot\prod_{\vepsilon\in\wt V_d}
\norm {f_\vepsilon}
\leq\norm h_{U(d)}\cdot\prod_{\vepsilon\in\wt V_d}
\norm {f_\vepsilon}_{2^{d-1}}
\end{align*}
by the Cauchy-Schwarz-Gowers Inequality and Inequality~\eqref{eq:boud_Norm}.  
\end{proof}

In particular, for $f\in L^{2^{d-1}}(\mu)$, we  have that
$ \norm{\CD_df}_{U(d)}^*\leq\norm f_{U(d)}^{2^d-1}$. On the other 
hand, by~\eqref{eq:norme_Dd}, 
$$
\norm f_{U(d)}^{2^d}= \langle\CD_d f;f\rangle\leq
\norm{\CD_d f}_{U(d)}^*\cdot\norm f_{U(d)}
$$ and thus $\norm{\CD_d f}_{U(d)}^*\geq\norm f_{U(d)}^{2^d-1}$. We 
conclude:
\begin{proposition}
\label{prop:norm_Df}
For every $f\in L^{2^{d-1}}(\mu)$, 
$\norm{\CD_d f}_{U(d)}^*=\norm f_{U(d)}^{2^d-1}$.
\end{proposition}

While the following proposition is not used in the 
sequel, it gives a helpful description of the anti-uniform space:
\begin{proposition}
\label{prop:convex_hull}
The unit ball of the anti-uniform space of level $d$ is the closed 
convex hull in $L^{2^{d-1}/(2^{d-1}-1)}(\mu)$ of the set
$$
 \bigl\{ \CD_d f\colon f\in L^{2^{d-1}}(\mu),\
 \norm f_{U(d)}\leq 1\bigr\}.
$$
\end{proposition}

\begin{proof}
The proof is a simple application of duality. 

Let $B\subset L^{2^{d-1}/(2^{d-1}-1)}(\mu)$  
be the unit ball 
of the 
anti-uniform norm $\norm\cdot_{U(d)}^*$. 
Let $K$ be the convex hull of 
the set in the statement and let $\overline K$ be its closure in 
$L^{2^{d-1}/(2^{d-1}-1)}(\mu)$.

By Proposition~\ref{prop:norm_Df}, for every $f$ with $\norm 
f_{U(d)}\leq 1$, we have 
$\CD_d f\in B$. Since $B$ is convex, $K\subset B$.
Furthermore, $B$ is contained in the unit ball of 
$L^{2^{d-1}/(2^{d-1}-1)}(\mu)$ and is a weak* compact subset of this space.
Therefore, $B$ is closed in $L^{2^{d-1}/(2^{d-1}-1)}(\mu)$ and 
$\overline K\subset B$.

We check that $\overline K\supset B$. If this does not hold, there 
exists 
$g\in  L^{2^{d-1}/(2^{d-1}-1)}(\mu)$ 
satisfying $\norm g_{U(d)}^*\leq 1$ and $g\notin\overline K$.
By the Hahn-Banach Theorem, there exists $f\in L^{2^{d-1}}(\mu)$ with 
$\langle f ; h\rangle\leq 1$ for every $h\in K$ and 
$\langle f;g\rangle> 1$. This last property implies that 
$\norm f_{U(d)}>1$.  
Taking $\phi=\norm f_{U(d)}\inv\cdot f$, we have that
$\norm\phi_{U(d)}=1$ and $\CD_d\phi\in K$.  
Thus by the first property of $f$, 
$\langle \CD_d\phi;f \rangle\leq 1$.  But 
$$
 \langle \CD_d\phi;f\rangle=\norm f_{U(d)}^{-2^d+1}
\langle\CD_d f;f\rangle= \norm f_{U(d)}
$$
and we have a contradiction.
\end{proof}

It can be shown that when $Z$ is finite, the set appearing in 
Proposition~\ref{prop:convex_hull} is already closed and convex:
\begin{proposition}
Assume $Z$ is finite.  Then the set 
$$
 \bigl\{ \CD_d f\colon \norm f_{U(d)}\leq 1\bigr\}
$$
is the unit ball of the anti-uniform norm. 
\end{proposition}

We omit the proof of this result, as the proof (for finite $Z$) 
 is similar to that of 
Theorem~\ref{th:k} below, which seems more useful.
For the general case, 
the analogous statement is not as clear because the ``uniform space'' does 
not consist only of functions.

\subsection{Approximation results for anti-uniform functions}

\label{subsec:approx}
\begin{theorem}
\label{th:k}
Assume $d\geq 1$ is an integer.  For every anti-uniform function $g$ with 
$\norm g_{U(d)}^*=1$, integer $k\geq d-1$, and $\delta>0$, 
the function $g$ can be written as
$$
 g=\CD_d f+h,
$$
where
\begin{gather*}
\norm f_{2^k}\leq 1/\delta;\\
\norm h_{2^k/(2^k-1)}\leq\delta;\\
\norm f_{U(d)}\leq 1.
\end{gather*}
\end{theorem}

As in the Dual Form of the Inverse Theorem, there is no hypothesis 
on $\norm g_\infty$ and we do not assume that the function $g$ is bounded.

\begin{proof}
Fix $k\geq d-1$ and $\delta>0$. 
For $f\in L^{2^k}(\mu)$, define
\begin{equation}
\label{eq:def_nnorm}
 \nnorm f=
\begin{cases}
\bigl(
\norm f_{U(d)}^{2^k}+\delta^{2^k}\norm f_{2^k}^{2^k}
\bigr)^{1/2^k}
& \text{if }k\geq d;\\
\bigl( 
\norm f_{U(d)}^{2^d}+\delta^{2^d}\norm f_{2^{d-1}}^{2^d}
\bigr)^{1/2^d}
& \text{if }k=d-1.
\end{cases}
\end{equation}

Since $\norm f_{U(d)}\leq\norm f_{2^{d-1}}\leq\norm f_{2^k}$
for every $f\in L^{2^k}(\mu)$, $\nnorm f$ is well defined on 
$L^{2^k}(\mu)$ and $\nnorm\cdot$ is a norm on this space, equivalent 
to the norm $\norm\cdot_{2^k}$.

Let $\nnorm\cdot ^*$ be the dual norm of $\nnorm\cdot$: for 
$g\in L^{2^k/(2^k-1)}(\mu)$,
$$
 \nnorm g^*=
\sup\Bigl\{
\bigl|\langle f; g\rangle\bigr|\colon
f\in L^{2^k}(\mu),\ \nnorm f\leq 1\Bigr\}.
$$

This dual norm is equivalent to the norm $\norm\cdot_{2^k/(2^k-1)}$.
Since $\nnorm f\geq\norm f_{U(d)}$ for every $f\in L^{2^k}(\mu)$, we have that 
$$
 \nnorm g^*\leq\norm g_{U(d)}^*\text{ for every }g.
$$

Fix an anti-uniform function $g$ with $\norm g_{U(d)}^*\leq 1$.
Since $\nnorm f\geq\norm f_{U(d)}$ for every $f$, we have that 
$$
 c:= \nnorm g^*\leq\norm g_{U(d)}^*\leq 1 .
$$
Set $g'=c\inv g$, and so $\nnorm {g'}^*=1$.

Since the norm $\nnorm\cdot$ is equivalent to the norm 
$\norm\cdot_{2^k}$ and the Banach space 
$(L^{2^k}(\mu),\norm\cdot_{2^k})$ is reflexive, 
the Banach space $(L^{2^k}(\mu),\nnorm\cdot)$ is also
reflexive. This means that $(L^{2^k}(\mu),\nnorm\cdot)$ 
is the dual of the Banach space
$(L^{2^k/(2^k-1)}(\mu),\nnorm\cdot^*)$. Therefore, there exists 
$f'\in L^{2^k}(\mu)$ with
$$
\nnorm {f'}=1\text{ and }\langle g';f'\rangle = 1.
$$
By definition~\eqref{eq:def_nnorm} of $\nnorm{f'}$, 
\begin{equation}
\label{eq:borne}
\norm{f'}_{U(d)}\leq 1\text{ and }\norm{f'}_{2^k}\leq 1/\delta.
\end{equation}

Assume first that $k\geq d$. (We explain the modifications 
needed for the case $k=d-1$ after.)

By~\eqref{eq:def-norm2}, \eqref{eq:norme_Dd}, and the symmetries of the
measure $\mu_d$,
for every $\phi\in L^{2^k}(\mu)$ and every $t\in\R$,
\begin{equation}
\label{eq:fth_Ud}
\norm{f'+t\phi}_{U(d)}^{2^d}=\norm {f'}_{U(d)}^{2^d}+2^dt\langle\CD_d 
f';\phi\rangle+o(t),
\end{equation}
where by $o(t)$ we mean any function such that $o(t)/t\to 0$ as $t\to 0$.

Raising this to the power $2^{k-d}$, we have that 
$$
\norm{f'+t\phi}_{U(d)}^{2^k}=
\norm {f'}_{U(d)}^{2^k}+2^kt\norm {f'}_{U(d)}^{2^k-2^d}
\langle\CD_d f';\phi\rangle+o(t).
$$
On the other hand,
$$
 \norm {f'+t\phi}_{2^k}^{2^k}=
\norm {f'}_{2^k}^{2^k}+2^kt\langle f'^{2^k-1};\phi\rangle+o(t).
$$
Combining these expressions and using the definition~\eqref{eq:def_nnorm}
of $\nnorm{f'+t\phi}$ and of $\nnorm {f'}$, 
we have that 
\begin{align*}
\nnorm{f'+t\phi}^{2^k}= & \norm{f'+t\phi}_{U(d)}^{2^k}+
\delta^{2^k}\norm{f'+t\phi}_{2^k}^{2^k}\\
= & \nnorm{f'}^{2^k}
+
2^kt\norm {f'}_{U(d)}^{2^k-2^d}
\langle\CD_d f';\phi\rangle
+
\delta^{2^k}2^kt\langle f'^{2^k-1};\phi\rangle+o(t)\\
= & 1+ 2^kt\norm {f'}_{U(d)}^{2^k-2^d}
\langle\CD_d f';\phi\rangle
+
\delta^{2^k}2^kt\langle f'^{2^k-1};\phi\rangle+o(t).
\end{align*}
Raising this to the power $1/2^k$, we have that
$$
\nnorm{f'+t\phi}=
1+t\norm {f'}_{U(d)}^{2^k-2^d}
\langle\CD_d f';\phi\rangle
+
\delta^{2^k}t\langle f'^{2^k-1};\phi\rangle+o(t).
$$
Since for every $\phi\in L^{2^k}(\mu)$ and every $t\in\R$ we have
$$
1+t\langle g';\phi\rangle= \langle g';f'+t\phi\rangle
\leq \nnorm{f'+t\phi},
$$
it follows that
$$
 1+t\langle g';\phi\rangle\leq
1+t\norm {f'}_{U(d)}^{2^k-2^d}
\langle\CD_d f';\phi\rangle
+
\delta^{2^k}t\langle f'^{2^k-1};\phi\rangle+o(t)\ .
$$
Since this holds for every $t$, we have
$$
\langle g';\phi\rangle =
\norm {f'}_{U(d)}^{2^k-2^d}
\langle\CD_d f';\phi\rangle
+
\delta^{2^k}\langle f'^{2^k-1};\phi\rangle.
$$
Since this holds for every $\phi$, we conclude that
$$ g'=\norm {f'}_{U(d)}^{2^k-2^d}\CD_d f'+ 
\delta^{2^k}f'^{2^k-1}.
$$
Thus
$$
g= c\norm {f'}_{U(d)}^{2^k-2^d}\CD_d f'+ c\delta^{2^k}f'^{2^k-1}.
$$
Set
$$
 f=\bigl(c\norm {f'}_{U(d)}^{2^k-2^d}\bigr)^{1/(2^d-1)} f'
\text{ and }
h=c\delta^{2^k}f'^{2^k-1}.
$$
Then 
$$
 g=\CD_d f+h
$$
and by~\eqref{eq:borne},
\begin{gather*}
 \norm f_{U(d)}\leq 1\ ;\ \norm f_{2^k}\leq 1/\delta \\
\norm h_{2^k/(2^k-1)}= 
c\delta^{2^k}\norm{f'}_{2^k}^{2^k-1}\leq\delta.
\end{gather*}

For the case $k=d-1$, 
for every $\phi\in L^{2^k}(\mu)$ and every $t\in\R$, we 
have~\eqref{eq:fth_Ud} and
$$
\norm {f'+t\phi}_{2^{d-1}}^{2^d}=
\norm {f'}_{2^{d-1}}^{2^d}+2^dt  \norm {f'}_{2^{d-1}}^{2^{d-1}} \langle 
f'^{2^{d-1}-1};\phi\rangle+o(t).
$$
Thus 
$$
 \nnorm{f'+t\phi}=1+t\langle\CD_d f';\phi\rangle+
\delta^{2^d}
\norm {f'}_{2^{d-1}}^{2^{d-1}} \langle 
f'^{2^{d-1}-1};\phi\rangle+o(t).
$$
As above, we deduce that 
$$
 g'=\CD_df'+\delta^{2^d}\norm {f'}_{2^{d-1}}^{2^{d-1}} 
f'^{2^{d-1}-1}.
$$
Taking
$$
 f=c^{1/(2^d-1)}f'\text{ and } h=c \delta^{2^d}\norm {f'}_{2^{d-1}}^{2^{d-1}} 
f'^{2^{d-1}-1},
$$
we have the statement.
\end{proof}

When $Z$ is finite, we can say more:
\begin{theorem}
\label{th:borne}
Assume that $Z$ is finite.  
Given a  function $g$ with $\norm g_{U(d)}^*=1$ and  $\delta>0$, 
the function $g$ can be written as
$$
 g=\CD_d f+h, 
$$
where
\begin{gather*}
\norm f_\infty\leq 1/\delta;\\
\norm h_1\leq\delta;\\
\norm f_{U(d)}\leq 1.
\end{gather*}
\end{theorem}
\begin{proof}
By Theorem~\ref{th:k}, for every $k\geq d-1$ we can write
$$
 g=\CD_d f_k+h_k, 
$$
where
$$
 \norm{f_k}_{2^k}\leq 1/\delta;\\
\norm{h_k}_{2^k/(2^k-1)}\leq\delta;\\
\norm{f_k}_{U(d)}\leq 1.
$$
Let $N=|Z|$.  Since $\norm{f_k}_{2^k}\leq 1/\delta$, we have that
$\norm{f_k}_\infty\leq N/\delta$.  In the same way,
$\norm{h_k}_\infty\leq N\delta$. 
By passing to a subsequence, since the functions are 
uniformly bounded we can therefore 
assume that $f_k\to f$ and that $h_k\to 
h$ pointwise as $k\to+\infty$.   Thus $\CD_df_k\to\CD_d f$ 
pointwise and so
$$
 g=\CD_d f+h.
$$
Since $\norm{f_k}_{U(d)}\to\norm f_{U(d)}$, it follows that $\norm 
f_{U(d)}\leq 1$.  For every $k\geq d-1$, we have that $\norm{h_k}_1\leq
\norm{h_k}_{2^k/(2^k-1)}\leq\delta$. Since $\norm{h_k}_1\to\norm 
h_1$, it follows that $\norm h_1\leq \delta$.  
For $\ell\geq k\geq d-1$, 
$$
\norm{f_\ell}_{2^k}\leq\norm{f_\ell}_{2^\ell}
\leq1/\delta.
$$
Taking the limit as $\ell\to+\infty$, we have that
$\norm f_{2^k}\leq 1/\delta$ for every $k\geq d-1$ and so
$\norm f_\infty\leq 1/\delta$.
\end{proof}

\begin{question}
Does  Theorem~\ref{th:borne} also hold when $Z$ 
is infinite?
\end{question}
We conjecture that the answer is positive,
 but the proof given does not carry through to this case.

\subsection{Applications}
Theorems~\ref{th:k} and~\ref{th:borne}
give insight into the $U(d)$ norm, connecting it to the 
classical $L^p$ norms.  For example, we have:
\begin{corollary}
Let $\phi$ be a function with $\norm \phi \leq 1$ and 
$\norm \phi_{U(d)}=\theta>0$. Then for every $p\geq 2^{d-1}$, 
there exists a function $f$ such that  
$\norm f_p\leq 1$ 
and $\langle \CD_{d}f;\phi\rangle > (\theta/2)^{2^d}$.  

If $Z$ is finite, there exists a function $f$ with $\norm 
f_\infty\leq 1$ and $\langle \CD_{d}f;\phi\rangle > (\theta/2)^{2^d}$.
\end{corollary}

\begin{proof}
It suffices to prove the result when $p= 2^k$ for some integer $k\geq 
d-1$.
There exists $g$ 
with $\norm g_{U(d)}^{*}=1$ and 
$\langle g;\phi\rangle = \theta$.  Taking $\delta = \theta/2$ in 
Theorem~\ref{th:k}, we 
have the first statement.  For the second statement, apply 
Theorem~\ref{th:borne}.
\end{proof}

Theorem~\ref{th:borne} leads to an  equivalent 
reformulation of the Inverse Theorem, without any explicit
reference to the Gowers norms. 
 For all $d\geq 1$ and $\delta > 0$, there exists a family of 
``$(d-1)$-step nilsequences 
  of bounded complexity'' whose convex hull $\CF''(d, \delta)$
satisfies:

\begin{dualinverse2}
\label{th:reformulated}
For every $\delta > 0$, every function $\phi$ on $\Z_{N}$
with $\norm \phi_\infty\leq 1$, 
the function $\CD_d\phi$  can be written as $\CD_d\phi = g +h$ 
with $g\in\CF''(d,\delta)$ and $\norm h_{1}\leq\delta$.
\end{dualinverse2}

\begin{proof}
We show that the statement is equivalent to the 
Dual Form of the Inverse Theorem. First assume the Dual Form.  Given 
$\phi$ with $\norm \phi_\infty\leq 1$, we have that
$\norm \phi_{U(d)}\leq 1$ and thus
$\norm{\CD_d\phi}_{U(d)}^{*}\leq 1$.
By the Dual Form, $\CD_{d}\phi = h+\psi$, 
where $h\in\CF'(d,\delta)$ and $\norm\psi_{1}\leq\delta$, 
which is exactly the Reformulated Version.

Conversely, assume the Reformulated Version.  Let 
$g\in B_{U(d)^{*}}(1)$.  
Then by Theorem~\ref{th:k},
$g = \CD_{d}h +\psi$, where $\norm h_\infty\leq 2/\delta$ 
and $\norm\psi_{1}\leq \delta/2$.
  Define $\CF' = \CF'(d,\delta)$ to be equal to 
$(2/\delta)^{2^d-1} \CF''(d,\eta)$, where $\eta$
is a positive constant to 
be defined later and
$\CF''(d,\eta)$ is as in the 
 Reformulated Version.  By the Reformulated Version, 
$\CD_{d}h = f+\psi$, with 
$$
f\in\CF'\text{ and }
\norm\psi_1\leq (2/\delta)^{2^d-1}\eta\  .
$$  
  Then 
$g = f+\phi+\psi$ with $f\in\CF'$ and 
$\norm{\phi+\psi}_{1}\leq \delta/2+(2/\delta)^{2^d-1}\eta$.
Taking $\eta=(\delta/2)^{2^d}$, we have the result.
\end{proof}

\subsection{Anti-uniformity norms and embeddings}
\label{sec:embedding}

This section is a conjectural, and somewhat optimistic, exploration 
of the possible
uses of the theory of anti-uniform norms we have 
developed.  The main interest is not the sketches of proofs 
included, but rather the questions posed and the
directions that we conjecture may be approached using these methods.

\begin{definition}
If $G$ is a $(d-1)$-step nilpotent Lie group and $\Gamma$ is a
discrete, cocompact subgroup of $G$, the compact manifold $X = G/\Gamma$ 
is {\em $(d-1)$-step nilmanifold}. 
The natural action of $G$ on $X$ by left translations is written as 
$(g,x)\mapsto g.x$ for $g\in G$ and $x\in X$.
\end{definition}

We recall  the following ``direct'' result (a converse
to the Inverse Theorem), proved along the lines of 
arguments in~\cite{HK1}: 
\begin{proposition}[Green and Tao (Proposition~12.6, \cite{GT2})]
\label{prop:12_6}
Let $X=G/\Gamma$ be a $(d-1)$-step nilmanifold,
 $x\in X$, $g\in G$, 
$F$ be a continuous function on  $X$, and $N\geq 2$ be an integer.
Let $f$ be a function on $\Z_N$ 
with $|f|\leq 1$. Assume that for some $\eta>0$, 
$$
 \bigl|\E_{0\leq n<N} f(n)F(g^n\cdot x)\bigr|\geq\eta.
$$
Then there exists a constant $c = c(X,F, \eta) > 0$ such that 
$$
 \norm f_{U(d)}\geq c.
$$
\end{proposition}

The key point is that the constant $c$ 
depends only on $X$,  $F$, and $\eta$, and not on
 $f$, $N$, $g$ or $x$.

\begin{remark}
In~\cite{GT2}, the average is taken over 
the interval $[-N/2,N/2]$ instead of $[0,N)$, but
the proof of Proposition~\ref{prop:12_6} is the same 
for the modified choice of interval.

A similar result  is given in Appendix G of~\cite{GTZ1}, and 
proved using simpler methods, 
but there the conclusion is about the norm 
 $\norm f_{U_d(\Z_{N'})}$, where $N'$ is sufficiently large with 
respect to $N$.
\end{remark}

By duality, Proposition~\ref{prop:12_6} can be rewritten as
\begin{proposition}
\label{prop:12_6bis}
Let $X=G/\Gamma, x,g,F$  be as in 
Proposition~\ref{prop:12_6}.
Let $N\geq 2$ be an integer and let  $h$ denote the 
function $n\mapsto F(g^n\cdot x)$ 
restricted to $[0,N)$ and considered as a function on $\Z_N$.
Then for  every $\eta>0$,  we can write
$$
h=\phi+\psi
$$
where $\phi$ and $\psi$
are functions on $Z_N$ with
$\norm{\phi}_{U(d)}^*\leq c(X,F,\eta)$ and
 $\norm{\psi}_1\leq\eta$.
\end{proposition}

Proposition~\ref{prop:12_6} does not imply that
$\norm{h}_{U(d)}^*$ is bounded independent of $N$, and 
using~\eqref{eq:U2dual}, one can easily construct a
counterexample for $d=2$ and $X=\T$. 
On the other hand, 
for $d=2$ we do have that $\norm{h}_{U(d)}^*$ 
is bounded independent of $N$ when 
the function $F$ is sufficiently smooth.
Recalling that the 
Fourier series of a continuously differentiable function on $\T$ is
absolutely convergent and directly computing using Fourier coefficients, we have:
\begin{proposition}
Let $F$ be a 
continuously differentiable function on $\T$ and let $\alpha\in\T$.
Let $N\geq 2$ be an integer and let  $h$ denote the restriction of the
function $n\mapsto F(\alpha^n)$ to $[0,N)$, considered as a function on $\Z_N$.
Then
$$
 \norm{h}_{U(2)}^*\leq c\norm{\wh F}_{\ell^1(\Z)},
$$
where $c$ is a universal constant.
\end{proposition}

A similar result holds for functions on $\T^k$.

It is natural to ask whether a similar result holds for 
$d>2$.  For the remained of this section, we assume that every nilmanifold $X$ is 
endowed with a smooth Riemannian metric. For $k\geq 1$,
we let $\CC^k(X)$ denote the space of $k$-times continuously 
differentiable functions on $X$, endowed with the usual norm
$\norm\cdot_{\CC^k(X)}$.  We ask if the dual 
norm is bounded independent of $N$:
\begin{question}
\label{qu:smooth_dual}
Let $X=G/\Gamma$ be a $(d-1)$-step nilmanifold. Does there exist 
an integer 
$k\geq 1$ and a positive constant $c$ such that for all choices 
of a function $F\in\CC^k(X)$, $g\in G$, $x\in X$
 and integer $N\geq 2$,  writing $h$ for the restriction to $[0,N)$ of 
the function $n\mapsto F(g^n\cdot x)$, considered as a function on $\Z_N$, 
we have 
 $$
 \norm {h}_{U(d)}^*\leq c\norm F_{\CC^k(X)}?
$$
\end{question}

\begin{definition}
If $g\in G$ and $x\in X$ are such that $g^N\cdot x=x$, we say that the 
map $n\mapsto g^n\cdot x$ is an \emph{embedding} of $\Z_N$ in $X$.
\end{definition}
\begin{proposition}
\label{prop:embed-ZN}
The answer to Question~\ref{qu:smooth_dual} is positive under the 
additional hypothesis that 
$n\mapsto g^n\cdot x$ is an embedding of $\Z_N$ in $X$,
that is, that 
$g^N\cdot x=x$. 
\end{proposition}

The proof of this proposition is similar to that of 
Proposition 5.6 in~\cite{HK3} and so we omit it.

More generally, we can phrase these results and the resulting 
question for groups other than 
$\Z_N$.  We restrict ourselves to the case of $\T$, 
as the extension to $\T^k$ is clear.  
By the same argument used for Proposition~\ref{prop:12_6}, we have:
\begin{proposition}
\label{prop:12_6_T}
Let $X=G/\Gamma$ be a $(d-1)$-step nilmanifold,
 $x\in X$, $u$ be an element in the Lie algebra of $G$, and 
$F$ be a continuous function on  $X$.
Let $f$ be a function on $\T$ 
with $|f|\leq 1$. Assume that for some $\eta>0$ we have
$$
\Bigl|\int f(t)F\bigl(\exp(tu)\cdot  x\bigr)\,dt\Bigr|\geq\eta,
$$
where we identify $\T$ with $[0,1)$ in this integral.  
Then there exists a constant $c = c(X,F, \eta) > 0$ such that 
$$
 \norm f_{U(d)}\geq c.
$$
\end{proposition}

By duality, Proposition~\ref{prop:12_6_T} can be rewritten as
\begin{proposition}
\label{prop:12_6_Tbis}
Let $X=G/\Gamma, x,u,F$, and $c=c(X,F,\eta)$ be as in 
Proposition~\ref{prop:12_6_T}.
Let $h$ denote the restriction of the
function $t\mapsto F\bigl(\exp(tu)\cdot x)$ to $[0,1)$, considered as a function on 
$\T$.
Then for every  $\eta>0$,  we can write
$$
h=\phi+\psi,
$$
where $\phi$ and $\psi$ are functions on $\T$ with
$\norm{\phi}_{U(d)}^*\leq c$ and  $\norm{\psi}_1\leq\eta$.
\end{proposition}

We can ask the analog of Question~\ref{qu:smooth_dual} for
the group $\T$:
\begin{question}
\label{qu:smooth_dual_T}
Let $X=G/\Gamma$ be a $(d-1)$-step nilmanifold. Does there exist 
an integer 
$k\geq 1$ and a positive constant $c$ such that for all choices of a 
function $F\in\CC^k(X)$, $u$ in the Lie algebra 
of $G$, and $x\in X$, writing $h$ for the restriction of 
the function $t\mapsto F\bigl(\exp(tu)\cdot x\bigr)$ to $[0,1)$,
considered a function on $\T$,
 we have
$$
 \norm {h}_{U(d)}^*\leq c\norm F_{\CC^k(X)}?
$$
\end{question}
Analogous to Proposition~\ref{prop:embed-ZN},
 the answer to this question is positive 
under the additional hypothesis that 
$t\mapsto\exp(tu)\cdot x$ is an 
\emph{embedding} of $\T$ in $X$, meaning that 
 $\exp(u)\cdot x=x$.

\section{Multiplicative  structure}
\subsection{Higher order Fourier Algebras}

In light of Theorem~\ref{th:k}, the
family of functions $g$ on $Z$ of the form $g=\CD_d f$ for
$f\in L^{2^k} (\mu)$ for some $k\geq d-1$ is relevant, and more generally,
cubic convolution products for functions $f_\vepsilon$, $\vepsilon\in\wt
V_d$,
belonging to $L^{2^k}(\mu)$ for some $k\geq d-1$. 
We only consider the case $k=d-1$,
as it gives rise to interesting algebras.

\begin{definition}
For an integer $d\geq 1$, define 
$A(d)$ to be the space of functions $g$ on $Z$ that can be written as
\begin{equation}
\label{eq:defAd}
 g(x)=\sum_{j=1}^\infty \CD_{d}(f_{j,\vepsilon}\colon\vepsilon\in\wt V_{d})
\end{equation}
where  all the functions $f_{j,\vepsilon}$ belong to 
$L^{2^{d-1}}(\mu)$ and
\begin{equation}
\label{eq:defAd2}
\sum_{j=1}^\infty\prod_{\vepsilon\in \wt V_{d}}
\norm{f_{j,\vepsilon}}_{2^{d-1}}<+\infty.
\end{equation}

For $g\in A(d)$, we define
\begin{equation}
\label{eq:defNormAd}
\norm g_{A(d)}=\inf 
\sum_{j=1}^\infty\prod_{\vepsilon\in \wt V_{d}}
\norm{f_{j,\vepsilon}}_{2^{d-1}},
\end{equation}
where the infimum is taken over all families of functions
$f_{j,\vepsilon}$ in $L^{2^{d-1}}(\mu)$ 
satisfying~\eqref{eq:defAd} and~\eqref{eq:defAd2}.
\end{definition}

We call $A(d)$ the \emph{Fourier algebra of order 
$d$}; we show in this section that it is a  Banach algebra.

It follows from the definitions that $A(1)$ consists of the 
constant functions with the norm $\norm{\cdot}_{A(1)}$ 
being absolute value.  
Clearly, if $Z$ is finite
and $d\geq 2$, then every function on $Z$ 
belongs to 
$A(d)$ and we can replace each series by a finite sum in the
definitions.

It is easy to check that $A(d)$ is a vector space of 
functions.

Furthermore, by~\eqref{eq:boud_DProduct} and 
Lemma~\ref{prop:continuous},
condition~\eqref{eq:defAd2} implies that the series
in~\eqref{eq:defAd} converges under the uniform norm and that 
every function in $A(d)$ is a continuous function on $Z$.
Moreover, by~\eqref{eq:boud_DProduct},
$\norm g_\infty\leq\norm g_{A(d)}$ and 
$\norm\cdot_{A(d)}$ is a norm on $A(d)$.

For every $g\in A(d)$, we have 
that $g$ belongs to that anti-uniform space of level $d$ and that 
$$\norm g_{U(d)}^*\leq\norm g_{A(d)}.$$

If $(g_n)_{n\in\N}$ is a sequence in $A(d)$ with 
$g=\sum_{n=1}^\infty\norm{g_n}_{A(d)}<+\infty$, then the series 
$\sum_{n=1}^\infty g_n$ converges under the uniform norm, the 
sum $g$ of this series belongs to $A(d)$, and the series 
converges to $g$ in $A(d)$. This shows that the space $A(d)$ endowed 
with the norm $\norm\cdot_{A(d)}$ is a Banach space.

Let $\CC(Z)$ denote the space of continuous functions 
on $Z$.  We summarize:
\begin{proposition}
\label{prop:Ad_continuous}
$A(d)$ is a linear subspace of $\CC(Z)$ and of the anti-uniform space 
of level $d$. For every $g\in A(d)$, we 
have that $\norm g_\infty\leq\norm g_{A(d)}$. The space $A(d)$ endowed 
with the norm $\norm\cdot_{A(d)}$ is a Banach space.
\end{proposition}

\subsection{Tao's uniform almost periodicity norms}
In~\cite{T}, Tao introduced a sequence of norms, 
the uniform almost periodicity norms, that also play a 
dual role to the Gowers uniformity norms:
\begin{definition}[Tao~\cite{T}]
For $f\colon Z\to \C$, define 
$\norm f_{\uap^0(Z)}$ to be 
equal to $|c|$ if $f$ is equal to the constant $c$, 
and to be infinite otherwise.  For $d\geq 1$, define $\norm f_{\uap^{d+1}(Z)}$ 
to be the infimum of all constants $M> 0$ such that for all $n\in\Z$,
$$T^nf = M\E_{h\in H}(c_{n,h}g_h),
$$
for some finite nonempty set $H$, collection of functions 
$(g_h)_{h\in H}$ from $Z$ to $\C$ satisfying $\norm{g_h}_{L^\infty(Z)}\leq 1$, 
collection of functions $(c_{n,h})_{n\in Z, h\in H}$ from
$Z$ to $\C$ satisfying $\norm{c_{n,h}}_{\uap^d(Z)}\leq 1$, and a
random variable $h$ taking values in $H$.  

When the underlying group is clear, we omit it from the notation 
and write $\norm f_{\uap^d(Z)} = \norm f_{\uap^d}$.  
\end{definition}

\begin{remark}
The definition given in~\cite{T} implicitly assumes that $Z$ 
is finite; to extend to the case that $Z$ is infinite, take 
$H$ to be an arbitrary probability space and view 
the functions $g_h$ and $c_{n,h}$ as random variables.
\end{remark}

Tao shows that this defines finite norms $\uap^d$ for $d\geq 1$ 
and that the uniformly almost periodic functions of order $d$ (meaning 
functions for which the $\uap^d$ norm is bounded) form a 
Banach algebra:
$$
\norm{fg}_{\uap^d} \leq \norm{f}_{\uap^d} \norm{g}_{\uap^d}.
$$

The $\uap^{d-1}$ and $A(d)$ norms are related: both 
are algebra norms and they satisfy similar properties, such as 
$$
\norm{f}_{\uap^{d-1}} \geq \norm{f}_{U(d)}^*
$$
and 
$$
\norm{f}_{A(d)}  \geq \norm{f}_{U(d)}^*.
$$
For $d=2$, the two norms are in fact the same (an exercise in~\cite{TV} 
due to Green and Section~\ref{sec:d=2} below). 
However, in general we do not know if they 
are equal:
\begin{question}
For a function $f\colon Z\to\C$, is 
$$
\norm{f}_{A(d)} = \norm{f}_{\uap^{d-1}}
$$
for all $d\geq 2$?
\end{question}

In particular, while the $\uap$ norms satisfy 
$$
\norm f_{\uap(d-1)} \geq \norm f_{\uap(d)}
$$ 
for all $d\geq 2$, we do not know if the same inequality 
holds for the norms $A(d)$.

\subsection{The case $d=2$}
\label{sec:d=2}
We give a further description for $d=2$, relating these 
notions to the classical objects in Fourier analysis.  

We have that $\wt V_2=\{01,10,11\}$.
Every function $g$ defined  as a cubic convolution product of $f_{\vepsilon}$, 
$\vepsilon\in\wt V_{2}$, satisfies
\begin{align}
\label{eq:2_3}
\sum_{\xi\in\wh Z}|\wh{g}(\xi)|^{2/3} = & 
\sum_{\xi\in\wh Z}\prod_{\vepsilon\in\wt V_{2}}|\wh{f_{\vepsilon}}(\xi)|^{2/3}\\
\nonumber
\leq & \prod_{\vepsilon\in\wt V_{2}}\bigl(\sum_{\xi\in\wh Z}|\wh{f_{\vepsilon}}(\xi)
|^{2}\bigr)^{1/3} = 
\prod_{\vepsilon\in\wt V_{2}}\norm{f_{\vepsilon}}_{L^{2}(\mu)}^{2/3}.
\end{align}
Thus
$$
 \sum_{\xi\in\wh Z}|\wh g(\xi)|
\leq \prod_{\vepsilon\in\wt V_2}\norm{f_\vepsilon}_{L^2(\mu)}.
$$
It follows that for $g\in A(d)$, we have that
$$
\sum_{\xi\in\wh Z}|\wh g(\xi)|\leq \norm g_{A(2)}.
$$

On the other hand, let  $g$ be a continuous function on $Z$ with
$\sum_{\xi\in\wh Z}|\wh g(\xi)|$ $<+\infty$. 
This function  can be written as (in this example, we 
make an exception to our convention that all functions are 
real-valued)
$$
 g(x)=\sum_{\xi\in\wh Z}\wh g(\xi)\,\xi(x)=
\sum_{\xi\in\wh Z}\wh g(\xi)\E_{t_1,t_2\in Z}\xi(x+t_1)
 \xi(x+t_2)\overline\xi(x+t_1+t_2).
$$
It follows that $g\in A(d)$ and $\norm g_{A(2)}\leq
\sum_{\xi\in\wh Z}|\wh g(\xi)|$.

We summarize these calculations:
\begin{proposition}
The space $A(2)$ coincides with the Fourier 
algebra $A(Z)$ of $Z$:
$$
A(Z):=\bigl\{g\in\CC(Z)\colon \sum_{\xi\in\wh Z}|\wh 
g(\xi)|<+\infty\bigr\}
$$
and, for $g\in A(Z)$, $\norm g_{A(2)}=\norm g_{A(Z)}$, 
which is equal 
by definition  to the sum of this series.
\end{proposition}

\subsection{$A(d)$ is an algebra of functions}
\begin{theorem}
\label{th:Ad_algebra}
The Banach space $A(d)$ is invariant under  pointwise multiplication and 
$\norm\cdot_{A(d)}$ is an algebra norm, meaning that 
for all $g,g'\in A(d)$, 
\begin{equation}
\label{eq:algebra}
\quad\ \norm{gg'}_{A(d)}\leq\norm g_{A(d)}\,\norm 
{g'}_{A(d)}.
\end{equation}
\end{theorem}

\begin{proof}
Assume that 
$$g(x)= \CD_{d}(f_{\vepsilon}\colon\vepsilon\in \wt V_{d})(x) \ \text{ and } \ 
g'(x)=\CD_{d}(f'_{\vepsilon}\colon\vepsilon\in \wt V_{d})(x),
$$
where $f_\vepsilon$ and $f'_\vepsilon\in L^{2^{d-1}}(\mu)$ for
every $\vepsilon\in V_d$.  Once we show that $gg'\in A(d)$ and
\begin{equation}
\label{eq:product}
 \norm{gg'}_{A(d)}\leq\prod_{\vepsilon\in V_d}\norm {f_\vepsilon}_{2^{d-1}}
\norm{f'_\vepsilon}_{2^{d-1}},
\end{equation}
the statement of the theorem follows from the definitions of the 
space $A(d)$ and its norm.

We have
$$
 g(x)g'(x)=
\E_{\vs\in\Z^d}\Bigl(\E_{\vt\in Z^d}
\prod_{\vepsilon\in\wt V_d}
f_\vepsilon(x+\vepsilon\cdot \vt)\;
f'_\vepsilon(x+\vepsilon\cdot \vs)
\Bigr)
$$

Writing $\vu=\vs-\vt$, we have that 
\begin{multline*}
 g(x)g'(x)=\E_{\vu\in\Z^d}\Bigl(\E_{\vt\in Z^d}
\prod_{\vepsilon\in\wt V_d}
f_\vepsilon(x+\vepsilon\cdot 
\vt)\;f'_\vepsilon(x+\vepsilon\cdot \vu+\vepsilon\cdot \vt)
\Bigr)\\
=\E_{\vu\in\Z^d}\Bigl(\E_{\vt\in Z^d}\prod_{\vepsilon\in\wt V_d}
\bigl(f_\vepsilon\;.\, f'_{\vepsilon,\vepsilon\cdot \vu})(x+\vepsilon\cdot \vt)
\Bigr)
=\E_{\vu\in\Z^d} g^{(\vu)}(x),
\end{multline*}
where
$$
 g^{(\vu)}(x):= \E_{\vt\in Z^d}\prod_{\vepsilon\in\wt V_d}
\bigl(f_\vepsilon\; .\, f'_{\vepsilon,\vepsilon\cdot \vu})(x+\vepsilon\cdot \vt).
$$

Then
\begin{multline*}
 \E_{\vu\in Z^d}\prod_{\vepsilon\in\wt V_d}
\norm{f_{\vepsilon}.f'_{\vepsilon, \vepsilon\cdot \vu}}_{2^{d-1}}\\
=\E_{u_1,\dots,u_{d-1}\in Z}
\Bigl(
\prod_{\substack{\vepsilon\in\wt V_d \\ \epsilon_d=0}}
\norm{f_\vepsilon\; .\,f'_{\vepsilon,\epsilon_1u_1+\dots+\epsilon_{d-1}u_{d-1}}}_{2^{d-1}}
\qquad\strut\\
\prod_{\substack{\vepsilon\in\wt V_d \\ \epsilon_d=1}}
\Bigl(
\E_{u_d\in Z}\norm{f_\vepsilon\;.\,
f'_{\vepsilon,\epsilon_1u_1+\dots+\epsilon_{d-1}u_{d-1}+u_d}}
_{2^{d-1}}^{2^{d-1}}
\Bigr)^{1/2^{d-1}}\Bigr).
\end{multline*}

But, for all $u_1,u_2,\dots,u_{d-1}\in Z$ and every
$\vepsilon\in\wt V_d$ with $\epsilon_d=1$,
\begin{multline*}
 \E_{u_d\in Z}
\norm{ f_\vepsilon\;.\,
f'_{\vepsilon,\epsilon_1u_1+\dots+\epsilon_{d-1}u_{d-1}+u_d}}
_{2^{d-1}}^{2^{d-1}}\\
=  \E_{v\in Z}\norm{ f_\vepsilon\;.\,f'_{\vepsilon,v} }_{2^{d-1}}^{2^{d-1}}
=\norm{f_\vepsilon}_{2^{d-1}}^{2^{d-1}}\;
\norm{f'_\vepsilon}_{2^{d-1}}^{2^{d-1}}.
\end{multline*}
On the other hand,
\begin{align*}
\E_{u_1,\dots,u_{d-1}\in Z} &
\prod_{\substack{\vepsilon\in\wt V_d \\ \epsilon_d=0}}
\norm{f_\vepsilon\;.\,
f'_{\vepsilon,\epsilon_1u_1+\dots+\epsilon_{d-1}u_{d-1}}}_{2^{d-1}}
\\
\leq \prod_{\substack{\vepsilon\in\wt V_d \\ \epsilon_d=0}} &
\Bigl(
\E_{u_1,\dots,u_{d-1}\in Z}
\norm{f_\vepsilon\;.\,f'_{\vepsilon,\epsilon_1u_1+\dots+\epsilon_{d-1}u_{d-1}
}}_{2^{d-1}}^{2^{d-1}}
\Bigr)^{1/2^{d-1}}.
\end{align*}

But, for $\vepsilon\in\wt V_d$ with $\epsilon_d=0$, we have that 
$\epsilon_1,\dots,\epsilon_{d-1}$ are not all equal to $0$ 
and
\begin{align*}
 \E_{u_1,\dots,u_{d-1}\in Z}
\norm{f_\vepsilon\;.\,f'_{\vepsilon,\epsilon_1u_1+\dots+\epsilon_{d-1}u_{d-1}
}}_{2^{d-1}}^{2^{d-1}}
= & \E_{w\in Z}\norm{f_\vepsilon\;.\,f'_{\vepsilon,w}}_{2^{d-1}}^{2^{d-1}}\\
= & \norm{f_\vepsilon}_{2^{d-1}}^{2^{d-1}}\;
\norm{f'_\vepsilon}_{2^{d-1}}^{2^{d-1}}.
\end{align*}
Combining these relations, we obtain that 
$$
 \E_{\vu\in Z^d} \prod_{\vepsilon\in\wt V_d}
\norm{f\;.\,f'_{\vepsilon\cdot \vu}} _{2^{d-1}}\leq
\prod_{\vepsilon\in V_d}\norm {f_\vepsilon}_{2^{d-1}}\;
\norm{f'_{\vepsilon}}_{2^{d-1}}.
$$
Therefore, for $\mu\times\dots\times\mu$-almost every $\vu\in Z^d$ 
and for every $\vepsilon\in\wt V_d$, the function
$f_\vepsilon.f'_{\vepsilon,\vepsilon\cdot \vu}$ belongs to 
$L^{2^{d-1}}(\mu)$.  It follows that 
for $\mu\times\dots\times\mu$-almost every $\vu\in Z^d$,
the function $g^{(\vu)}$ 
belongs to $A(d)$ and that 
$$
 \E_{\vu\in Z^d}\norm{g^{(\vu)}}_{A(d)}\leq 
\prod_{\vepsilon\in V_d}\norm {f_\vepsilon}_{2^{d-1}}\;
\norm{f'_\vepsilon}_{2^{d-1}}.
$$
Since $gg'(x)=\E_{\vu \in Z^d}g^{(\vu)}(x)$, Inequality~\eqref{eq:product} follows.
\end{proof}

\subsection{Decomposable functions on $Z_d$}

Recall that $Z_{d}$ is the subset of $Z^{2^d}$ defined in~\eqref{eq:Zd} 
and the elements $\bx\in Z_{d}$ are written as 
$\bx = (x_\vepsilon\colon \vepsilon\in
V_d)$.

\begin{definition}
The space $D(d)$ of \emph{decomposable functions} consists in 
functions $F$ on $Z_d$ that can be written as
\begin{equation}
\label{eq:def_decomp}
F(\bx)= \sum_{j=1}^\infty\prod_{\vepsilon\in V_d}f_{j,\vepsilon}
(x_\vepsilon),
\end{equation}
where all the functions $f_{j,\vepsilon}$ belong to 
$L^{2^d}(\mu)$ and
\begin{equation}
\label{eq:def_decomp2}
 \sum_{j=1}^\infty\prod_{\vepsilon\in V_d}
\norm{f_{j,\vepsilon}}_{L^{2^d}(\mu)}<+\infty.
\end{equation}

For $F\in D(d)$, define
$$
 \norm F_{D(d)}=\inf\sum_{j=1}^\infty\prod_{\vepsilon\in V_d}
\norm{f_{j,\vepsilon}}_{2^d},
$$
where the infimum is taken over all families of functions 
$f_{j,\vepsilon}$ in $L^{2^d}(\mu)$ satisfying~\eqref{eq:def_decomp} 
and~\eqref{eq:def_decomp2}.
\end{definition}

By the remark 
following~\eqref{eq:boud_Norm}, a function $F\in D(d)$ belongs to $L^2(\mu_d)$ and
$$
\norm F_{L^2(\mu_d)}\leq 
\norm F_{D(d)}.
$$

Clearly, if $Z$ is finite, then every function on $Z_d$ belongs to 
$D(d)$ and in the definition, we can replace the 
series by a finite sum.  

We summarize the properties of the space $D(d)$:
\begin{proposition}
$D(d)$ is a linear subspace of $L^2(\mu_d)$ and for $F\in D(d)$, we 
have that $\norm F_{L^2(\mu_d)}\leq\norm F_{D(d)}$. The space 
$D(d)$ endowed with the norm $\norm\cdot_{D(d)}$ is a Banach space.
\end{proposition}

\subsection{Diagonal translations}

\begin{definition}
For $t\in Z$, we write $t^\Delta=(t,t,\dots,t)\in Z_d$.
The map $\bx\mapsto \bx+t^\Delta$ is called the \emph{diagonal 
translation by $t$}.

Let $\CI(d)$ denote the subspace of $L^2(\mu_d)$ consisting of 
functions invariant under all diagonal translations. The orthogonal 
projection $\pi$ on $\CI(d)$ is given by
$$
 \pi F(\bx)=\E_{t\in Z}F(\bx+t^\Delta).
$$
\end{definition}

\begin{proposition}
\label{prop:proj_pi}
If $F$ belongs to $D(d)$, then $\pi F$ belongs to $D(d)$ and
$\norm{\pi F}_{D(d)}\leq\norm F_{D(d)}$. Furthermore, 
$\pi F$ is a continuous function on $Z_d$ satisfying
$\norm{\pi F}_\infty\leq \norm F_{D(d)}$.

In particular, functions $F$ belonging to $D(d)\cap\CI(d)$ are continuous 
on $Z_d$ and satisfy
$\norm{F}_\infty\leq \norm F_{D(d)}$.
\end{proposition}

\begin{proof}
Assume that $f$ is given by~\eqref{eq:def_decomp} where the functions 
$f_{j,\vepsilon}$ belong to $L^{2^d}(\mu)$ 
and~\eqref{eq:def_decomp2} is satisfied. Then
\begin{align*}
\pi F(\bx) 
&=\E_{t\in Z}\sum_{j=1}^\infty \prod_{\vepsilon\in V_d} 
f_{j,\vepsilon,t}(x_\vepsilon)\\
&=\E_{t\in Z}\sum_{j=1}^\infty \prod_{\vepsilon\in V_d} 
f_{j,\vepsilon,x_\vepsilon}(t).
\end{align*}
The first equality gives the first part of the proposition and the second 
implies the second part.
\end{proof}

\begin{theorem}
\label{th:decomposable}
For $F\in D(d)$ and $G\in D(d)\cap\CI(d)$, 
we have that $FG$ belongs to  $D(d)$ and that 
$\norm{FG}_{D(d)}\leq\norm F_{D(d)}\norm G_{D(d)}$.

In particular, $ D(d)\cap\CI(d)$, endowed with pointwise 
multiplication and the norm $\norm\cdot_{B(d)}$, is
a Banach algebra.
\end{theorem}

\begin{proof}
Since $\pi G=G$ when $G\in D(d)\cap \CI(d)$, it suffices to show that
for all $F,G\in D(d)$,  we have $F.\pi(G)\in D(d)$ and
\begin{equation}
\label{eq:product1}
\norm{F.\pi G}_{D(d)}\leq \norm F_{D(d)}\norm G_{D(d)}.
\end{equation}

First consider the case that  $F$ and $G$ are
 product function:
$$
F(\bx)=\prod_{\vepsilon\in V_d}f_\vepsilon(x_\vepsilon),\
G(\bx)=\prod_{\vepsilon\in V_d}g_\vepsilon(x_\vepsilon),
$$
where $f_\vepsilon$ and $g_\vepsilon\in L^{2^d}(\mu)$ for every 
$\vepsilon\in V_{d}$.
Then
$$
 (F.\pi G)(\bx)=\E_{t\in Z}\prod_{\vepsilon\in V_d}
(f_\vepsilon.g_{\vepsilon,t})(x_\vepsilon)=E_{t\in Z} H^{(t)}(\bx)\ ,
$$
where 
$$
H^{(t)}(\bx)\prod_{\vepsilon\in V_d}
(f_\vepsilon.g_{\vepsilon,t})(x_\vepsilon)\ .
$$
Furthermore,
$$
 \E_{t\in Z}\prod_{\vepsilon\in V_d}
\norm{f_\vepsilon.g_{\vepsilon,t}}_{2^d}
\leq \prod_{\vepsilon\in V_d}\Bigl(
\E_{t\in Z}\norm{f_\vepsilon.g_{\vepsilon,t}}_{2^d}^{2^d}
\Bigr)^{1/2^d}
=\prod_{\vepsilon\in V_d}\norm{f_\vepsilon}_{2^d}\norm{g_\vepsilon}_{2^d}\ .
$$
Thus for $\mu$-almost every $t\in Z$, we have that 
$f_\vepsilon.g_{\vepsilon,t}$ belongs to $L^{2^d}$ for every $\vepsilon$
and the function $H^{(t)}$  belongs to $B(d)$. 
Finally, 
$$
 \norm{F.\pi G}_{B(d)}\leq E_{t\in Z}\norm{H^{(t)}}_{B(d)}
\leq  E_{t\in Z}\prod_{\vepsilon\in V_d}
\norm{f_\vepsilon}_{2^d}\norm{g_\vepsilon}_{2^d}
$$
and the statement of the theorem follows from the definitions of the 
space $D(d)$ and its norm.
\end{proof}

\section{A result of finite approximation}

\subsection{A decomposition theorem}
For a probability space $(X,\mu)$, we assume throughout 
that it belongs to one of the two following classes:
\begin{itemize}
\item
$\mu$ is nonatomic. We refer to this case as the \emph{infinite case}.
\item 
$X$ is finite and $\mu$ is the uniform probability measure on $X$. 
We refer to this case as the \emph{finite case}.
\end{itemize}
This is not a restrictive assumption: Haar measure on a compact abelian group always falls into 
one of these two categories.

As usual, all subsets or partitions of $X$ are implicitly assumed to 
be measurable.

\begin{definition}
\label{def:almost-uniform}
Let $m\geq 2$ be an integer and let $(X_1,\dots,X_m)$ a partition of 
the probability space $(X,\mu)$. 
This partition is \emph{almost uniform} if:
\begin{itemize}
\item in the infinite case, $\mu(X_i)=1/m$ for every $i$.
\item In the finite case, $|X_i|=\lfloor |X|/m\rfloor$ or $\lceil 
|X|/m \rceil$ for every $i$.
\end{itemize}
\end{definition}

The main result of this paper is:
\begin{theorem}
\label{th:main}
Let $d\geq 1$ be an integer and let $\delta>0$.  There exists an integer 
$M=M(d,\delta)\geq 2$ and 
a constant $C=C(d,\delta)>0$ such that the following holds:
if $f_\vepsilon$, $\vepsilon\in\wt V_{d+1}$, are $2^{d+1}-1$ functions 
belonging to $L^{2^d}(\mu)$ with 
$\norm{f_\vepsilon}_{L^{2^d}(\mu)}\leq 1$ and 
$$
 \phi(x)=
 \CD_{d+1}(f_{\vepsilon}\colon\vepsilon\in \wt V_{d+1})(x),
$$
then for every $\delta>0$ there exist an almost 
uniform partition $(X_1,\dots,X_m)$ of $Z$ with $m\leq M$ 
sets, a nonnegative function $\rho$ on $Z$, 
and for $1\leq i\leq m$ and every $t\in Z$, a function 
$\phi_{i}^{(t)}$ on $Z$ such that
\begin{enumerate}
\item $\dis \norm\rho_{L^2(\mu)}\leq\delta$;
\item
$\norm{\phi_{i}^{(t)}}_\infty\leq 1$ and 
$\dis \norm{\phi_{i}^{(t)}}_{A(d)}\leq C$ for every $i$ and every $t$;
\item
\begin{equation}
\label{eq:main}
\Bigl| \phi(x+t)-\sum_{i=1}^m1_{X_i}(x)\phi_{i}^{(t)}(x)
\Bigr|\leq\rho(x)
\text{ for all }x,t\in Z.
\end{equation}
\end{enumerate}
\end{theorem}

Combining this theorem with an approximation result, this gives 
insight into properties of the dual norm.

\begin{remark}
In fact we show a bit more: 
each function $\phi_{i}^{(t)}$ is the sum of a bounded number of functions 
that are cubic convolution products of functions with $L^{2^{d-1}}(\mu)$ norm 
bounded by $1$.  
\end{remark}

\begin{remark}
The function $\phi$ in the statement of 
Theorem~\ref{th:main} satisfies 
$|\phi|\leq 1$ and thus $0\leq \rho \leq 2$.

Furthermore, the function $\phi$ belongs to $A(d+1)$, with 
$\norm\phi_{A(d+1)}\leq 1$. But Theorem~\ref{th:main} can not be 
extended to all functions belonging to $A(d+1)$, even for $d=1$.
\end{remark}

\begin{remark}
Theorem~\ref{th:main} holds for $d=1$, 
keeping in mind that $A(1)$ consists of constant functions and that 
$\norm\cdot_{A(1)}$ is the absolute value. 
In this case, the results can be proven directly and we sketch 
this approach.  In Section~\ref{sec:d=2}, we showed 
that the Fourier coefficients of the function $\phi$ satisfy
$$\sum_{\xi\in\wh Z}|\wh{\phi}(\xi)|^{2/3}\leq 1.
$$
Let $\psi$ be the trigonometric polynomial obtained by 
removing the Fourier coefficients in $\phi$ that are less than 
$\delta^{3}$.  The error term satisfies $\norm{\phi-\psi}_{\infty}\leq\delta$
and so the function $\rho$ in the theorem can be taken 
to be the constant $\delta$.   
There are at most $1/\delta^{2}$ characters so that $\xi$ 
such that $\wh{\psi}(\xi)\neq 0$.  
Taking a finite partition such that 
each of these characters is essentially constant 
on each set in the partition, we have that for every $t$ 
the function $\phi_{t}$ is essentially constant on each 
piece of the partition.  
\end{remark}

Before turning to the proof, we need some definitions, notation, and 
further results.  
Throughout the remainder of this section, we assume that an integer $d\geq 1$ 
is fixed, and the dependence of all constants on $d$ is 
implicit in all statements.  For notational convenience, 
we study functions belonging to $A(d+1)$ instead of 
$A(d)$.

\subsection{Regularity Lemma}
\begin{definition}
\label{def:partition}
Fix an integer $D\geq 2$. Let $(X,\mu)$ be a probability space of one 
of the two types considered in Definition~\ref{def:almost-uniform}.

Let $\nu$ be a measure on $Z^D$ such that each of its projections on 
$Z$ is equal to $\mu$.

Let $\CP$ be a partition of $Z$.
An atom of the product partition $\CP\times\ldots\times \CP$
($D$ times) of $Z^D$ is called a {\em rectangle} of $\CP$.  

A {\em $\CP$-function} on $Z^D$ is a 
function $f$ that is constant on each
rectangle of $\CP$.

For a function $F$ on $Z^D$, we define $F_\CP$ to be 
the $\CP$-function 
obtained by averaging over each rectangle with respect to the 
measure $\nu$: for every $x\in Z^D$, if $R$ is the rectangle containing 
$x$, then 
$$F_\CP(x) =
\begin{cases}
 \dis\frac{1}{\nu(R)}\int F\,d\nu & \text{ if }\nu(R)\neq 0;\\

0 & \text{ if }\nu(R)=0.
\end{cases}
$$

An $m$-step function is a $\CP$-function for some partition $\CP$ into 
at most $m$ sets. 
\end{definition}

As with $d$, we assume that the integer $D$ is fixed 
throughout and omit the explicit dependencies of the 
statements and constants on $D$.

We make use of the following version of the Regularity 
Lemma, a modification of the analytic version of 
Szemer\'edi's Regularity Lemma in~\cite{Lovasz}: 
\begin{theorem}[Regularity Lemma, revisited]
\label{th:regularity}
For every $D$ and  $\delta > 0$, there exists 
$M=M(D,\delta)$ such that if 
$(X,\mu)$ and $\nu$ are as in Definition~\ref{def:partition},
then for every function
$F$ on $Z^D$ with $|F|\leq 1$, there is an almost uniform
partition $\CP$ of $Z$ into $m\leq M$ sets  such that
for every $m$-step function  $U$ on $Z^D$ with $|U|\leq 1$, 
$$
\Bigl|\int U(F-F_\CP)\,d\nu \Bigr|\leq\delta  \ .
$$
\end{theorem}

We defer the proof to Appendix~\ref{appendix:regularity}. 
In the remainder of this section, we carry out the proof 
of Theorem~\ref{th:main}.

\subsection{An approximation result for decomposable functions}
We return to our usual definitions and notation.
We fix $d\geq 1$ and apply the Regularity Lemma to the probability 
space $(Z,\mu)$, $D=2^d$ and the probability measure $\mu_d$ on 
$Z^{2^d}$.

In this section, we show an approximation result that allows 
to go from weak to strong approximations:
\begin{proposition}
\label{th:RL2}
Let $F$ be a function on $Z_d$ belonging to $D(d)$ with $\norm 
F_{D(d)}\leq 1$ and $\norm F_\infty\leq 1$.
Let $\theta >0$ and 
 $\CP$ be the partition of $Z$ 
associated to $F$ and $\theta$ by the Regularity Lemma 
(Theorem~\ref{th:regularity}).
Then there exist constants $C=C(d)>0$
and $c=c(d)>0$  such that 
$$
 \norm {F-F_\CP}_2\leq (C\theta^c+\theta)^{1/2}\ .
$$
\end{proposition}

We first prove a result that allows us to 
pass from sets to functions:
\begin{lemma}
\label{lem:prod}
Assume that $F$ is a function  on $Z_d$ with $\norm F_\infty\leq 
1$.  Let $\theta >0$ and let $\CP$ be the partition of $Z$ 
associated to $F$ and $\theta$ by the Regularity Lemma 
(Theorem~\ref{th:regularity}).
If $f_\vepsilon$, $\vepsilon\in V_d$, are functions on $Z$ satisfying
$\norm{f_\vepsilon}_{2^d}\leq 1$ for every $\vepsilon$, then 
\begin{equation}
\label{eq:prod}
 \Bigl|\E_{\bx\in Z_d}(F-F_\CP)(\bx)
\prod_{\vepsilon\in V_d}f_\vepsilon(x_\vepsilon)\Bigr|
\leq C\theta^c,
\end{equation}
where $c=c(d)$ and $C=C(d)$ are positive constants.
\end{lemma}

In other words, writing $\norm\cdot_{D(d)}^*$ for the dual norm of 
the norm $\norm\cdot_{D(d)}$, we have that 
$$
\norm {F-F_\CP}_{D(d)}^*\leq C\theta^c.
$$

\begin{proof}
By construction, 
$\CP$ is  an almost uniform partition of $Z$ into $m<M(\eta)$ 
pieces and the function $F=F_\CP$ satisfies
\begin{equation}
\label{eq:regul}
|\E_{Z_d}U(F-F_\CP)|\leq\eta
\end{equation}
for every $m$-step function $U$ on $Z_d$ with $|U|\leq 1$.
We show~\eqref{eq:prod}.

By possibly changing the constant $C$, we can further assume that the 
functions $f_\vepsilon$ are all non-negative.  
Let $\eta>0$ be a parameter, with its value to be determined.  For 
$\vepsilon\in\{0,1\}^d$, set
$$
f'_\vepsilon(x)=\min\bigl(f_\vepsilon(x),\eta\bigr)
\text{ and }
f''_\vepsilon(x)=f_\vepsilon-f'_\vepsilon(x).
$$
Thus the average of~\eqref{eq:prod} can be written as a sum of 
$2^d$ averages, which we deal with separately.  

\noindent\textbf{a)} We first show that 
\begin{equation}
\label{eq:maj1}
\Bigl|\E_{\bx\in Z_d}(F-F_\CP)(\bx)
\prod_{\vepsilon\in  V_d}f'_\vepsilon(x_\vepsilon)\Bigr|\leq 
\eta^{2^d}\theta.
\end{equation}

For $u\in\R_+$, write
$$
A(\vepsilon,u)=\{x\in Z\colon f_\vepsilon(x)\leq u\}.
$$
For each $\vepsilon\in \{0,1\}^d$, we have that
$$
 f'_\vepsilon(x)=\int_0^\eta \one_{A (\vepsilon,u)}(x)\,du
$$
and so the average of the left hand side of~\eqref{eq:maj1} 
is the integral over $\bu=(u_\vepsilon\colon\vepsilon\in V_d)\in 
[0,\eta]^{2^d}$ of 
$$
 \E_{\bx\in Z_d}(F-F_\CP)(\bx)\prod_{\vepsilon\in V_d}\one_{A 
(\vepsilon,u_\vepsilon)}(x_\vepsilon).
$$
By~\eqref{eq:regul}, for each $\bu\in 
[0,\eta]^{2^d}$, the absolute value of this average is bounded by 
$\theta$.  Integrating, we have the bound~\eqref{eq:maj1}.

\medskip\noindent\textbf{b)} 
Assume now that for each $\vepsilon\in \{0,1\}^d$, the function
$g_\vepsilon$ is equal either to $f'_\vepsilon$ or to $f''_\vepsilon$, 
and that there exists $\valpha\in \{0,1\}^d$  with 
$g_\valpha=f''_\vepsilon$. We show that 
$$
 \Bigl|\E_{\bx\in Z_d}(F-F_\CP)(\bx)
\prod_{\vepsilon\in V_d}g_\vepsilon(x_\epsilon)\Bigr|
\leq 2\eta^{-2^d+1}.
$$
Since $|F-F_\CP|\leq 2$ and the functions $g_\vepsilon$ are 
nonnegative, it suffices to show that
$$
 \E_{\bx\in Z_d}
\prod_{\vepsilon\in \{0,1\}^d}g_\vepsilon(x_\vepsilon)\leq 
\eta^{-2^d+1}.
$$
By Corollary~\ref{cor:alpha}, the left hand side is bounded by
\begin{multline*}
 \prod_{\substack{\vepsilon\in V_d\\ \vepsilon\neq\valpha}}
\norm{g_\vepsilon}_{L^{2^{d-1}}(\mu)}\cdot\norm {g_\valpha}_{L^1(\mu)}
\leq \norm {g_\valpha}_{L^1(\mu)}
=\int 1_{f_\valpha>\eta}(x) f_\valpha(x)\\
\leq 
\norm{f_\valpha}_{2^d}
\mu\{x\in Z\colon f_\valpha(x)\geq\eta\}^{(2^d-1)/2^d}
\leq \eta^{-2^d+1}
\end{multline*}
and we have the statement.

\medskip\noindent\textbf{c)}  The left hand side of~\eqref{eq:prod} 
is thus bounded by 
$$\eta^{2^d}\theta+2(2^d-1)\eta^{-2^d+1}.
$$
Taking $\eta=\theta^{-1/(2^{d+1}-1)}$, we have the bound~\eqref{eq:prod}.
\end{proof}

We now use this to prove the proposition:
\begin{proof}[Proof of Proposition~\ref{th:RL2}]
Since $F$ belongs to $D(d)$ with $\norm F_{D(d)}\leq 1$, it follows
from the definition of this norm and from Lemma~\ref{lem:prod} that
$|\E_{\bx\in Z_d}(F-F_\CP)(\bx)F(\bx)|\leq C\theta^c$.

On the other hand, $F_\CP$ is an $m$-step function and by the property 
of the partition $\CP$ given by Theorem~\ref{th:regularity}, we have 
that
$|\E_{\bx\in Z_d}(F-F_\CP)(\bx)F_\CP(\bx)|\leq\theta$.
Finally,
$\E_{\bx\in Z_d}\bigl((F-F_\CP)(\bx)^2\bigr)
\leq C\theta^c+\theta$.
\end{proof}


\subsection{Proof of Theorem~\ref{th:main}}

We use the notation and hypotheses from the statement of 
Theorem~\ref{th:main}.

\subsection*{a) A decomposition}
Define $\Proj\,\,\colon L^1(\mu_d)\to L^1(\mu)$ to be the operator of 
conditional expectation. The most convenient definition of this 
operator is by duality: 
for $h\in L^\infty(\mu)$ and $H\in L^1(\mu_d)$, 
$$
\int_Z h(x)\, \Proj H(x)\,d\mu(x)=
\int_{Z_d} h(x_\bzero) H(\bx)\,d\mu_d(\bx).
$$
Recall that $\norm{\Proj H}_{L^{1}(\mu_{d}}\leq\norm H_{L^1(\mu_d)}$. 

By definition, when
$$
 H(x)=\prod_{\vepsilon\in V_d}f_\vepsilon(x_\vepsilon),
$$
where the functions $f_\vepsilon$ belong to $L^{2^{d-1}(\mu)}$, then
\begin{equation}
\label{eq:PH_Product}
 \Proj H(x)=\E_{\vt\in Z_d}\prod_{\vepsilon\in V_d} 
f_\vepsilon(x+\vepsilon\cdot\vt).
\end{equation}

For $\bx\in Z_d$, define
\begin{gather*}
G(\bx)=\bigotimes_{\vepsilon\in\wt V_d} f_{\vepsilon 0}(\bx)=
\prod_{\vepsilon\in \wt V_d}f_{\vepsilon 0}(x_\vepsilon)\\
F(\bx)
=\Bigl(\pi\bigotimes_{\vepsilon \in V_d} f_{\vepsilon 1}\Bigr)
(\bx)
=\E_{u\in Z}\prod_{\vepsilon\in V_d}
f_{\vepsilon 1}(x_\vepsilon+u).
\end{gather*}

For $x\in Z$, we have
\begin{align*}
 \phi(x)= & \E_{\vs\in Z_d}\prod_{\vepsilon\in\wt V_d}
\Bigl(
f_{\vepsilon 0}(x+\vepsilon\cdot \vs)\,\E_{u\in Z}
\prod_{\vepsilon\in V_d}f_{\vepsilon 1}(x+\vepsilon\cdot \vs+u)
\Bigr)\\
= & \Proj(G\cdot F).
\end{align*}

Recall that for $t\in Z$, $\phi_t$ is 
the function on $Z$ defined by $\phi_t(x)=\phi(x+t)$.

For $t\in Z$ and $\bx\in Z_d$, define
$$
 G_{t^\Delta}(\bx)=G(x+t^\Delta)=
\prod_{\vepsilon\in \wt V_d}f_{\vepsilon 0}(x_\vepsilon+t).
$$
Since the function $F$ is invariant under diagonal translations, for 
$x,t\in Z$ we have that
$$
 \phi_t(x)=\Proj (G_{t^\Delta}\cdot F)(x).
$$
By Proposition~\ref{prop:proj_pi}, the function $F$ belongs to $D(d)$
and $\norm F_{D(d)}\leq 1$.  Thus $\norm F_\infty\leq 1$.

Let $\delta>0$. Let $c$ and $C$ be as in Proposition~\ref{th:RL2} and 
let $\theta>0$ be such that $(C\theta^c+\theta)^{1/2}<\delta$.
Let  $\CP$ and $F_\CP$ be associated to $F$ and 
$\theta$ as in the Regularity Lemma.
Let $\CP=(A_1,\dots, A_m)$.

For $x,t\in Z$, we have that 
$$
  \phi_t(x)=\Proj (G_{t^\Delta}\cdot (F-F_\CP))+ 
\Proj(G_{t^\Delta}\cdot F_\CP)
$$
and we study the two parts of this sum separately.

\subsection*{b) Bounding the rest}
Define
$$
 \rho(x)= \bigl(\Proj(F-F_\CP)^2\bigr)^{1/2}.
$$
We have that 
$$\norm{\rho}_{2}= \norm{\Proj(F-F_{\CP})^{2}}_{L^{2}(\mu_{d})}^{1/2}
\leq\norm{(F-F_{\CP})^{2}}_{L^{1}(\mu_{d})}^{1/2} = 
\norm{F-F_\CP}_{L^{2}(\mu_{d})}\leq\delta, 
$$
where the last inequality follows from Proposition~\ref{th:RL2}.

Moreover, 
$$
\bigl| \Proj (G_{t^\Delta}\cdot (F-F_\CP))\bigr|\leq
\bigl(\Proj (G_{t^\Delta}^2)\bigr)^{1/2}\cdot
\bigl(\Proj(F-F_\CP)^2\bigr)^{1/2}\leq \rho(x)
$$
by~\eqref{eq:PH_Product} and Lemma~\ref{lem:Ddf}.

\subsection*{c) The main term}
We write elements of $\{1,\dots,m\}^{2^d}$ as
$$
 \bj=(j_\vepsilon\colon\vepsilon\in V_d).
$$
For $\bj=(j_\vepsilon\colon \vepsilon\in V_d)\in
\{1,\dots,m\}^{2^d}$, write
$$
 R_{\bj}=\prod_{\vepsilon\in V_d} A_{j_\vepsilon}.
$$
The function $F_\CP$ is equal to a constant on each 
rectangle $R_\bj$.
Let $c_\bj$ be this constant. We have that $|c_{\bj}|\leq 1$.

For $1\leq i\leq m$ and $t,x\in Z$, define
$$
 \phi_i^{(t)}(x):=\E_{\vs\in Z^d}
\sum_{\substack{ \bj\in\{1,\dots,m\}^{2^d} \\
j_\bzero = i}}
c_\bj
\prod_{\vepsilon\in \wt V_d}
1_{A_{j_\vepsilon}}(x+\vepsilon\cdot \vs).
\phi_{\vepsilon 0}(x+\vepsilon\cdot \vs).
$$
Since distinct rectangles are disjoint, it follows that 
$$
 \Bigl|\sum_{\substack{ \bj\in\{1,\dots,m\}^{2^d} \\
j_\bzero = i}}
c_\bj
\prod_{\vepsilon\in \wt V_d}
1_{A_{j_\vepsilon}}(x+\vepsilon\cdot \vs).
\phi_{\vepsilon 0}(x+\vepsilon\cdot \vs)\Bigr|
\leq 
\prod_{\vepsilon\in \wt V_d}
|\phi_{\vepsilon 0}(x+\vepsilon\cdot\vs)|.
$$
Thus
$$
 |\phi_i^{(t)}(x)|\leq 1.
$$
On the other hand, the function $\phi_i^{(t)}$ is the sum of
$m^{2^d-1}$ functions belonging to $A(d)$ with norm $\leq 1$ and thus
$$
 \norm{\phi_i^{(t)}}_{A(d)}\leq C=M^{2^d-1}.
$$

We claim that
\begin{equation}
\label{eq:FP_phiit}
P(G_{t^\Delta}\cdot\CF_\CP)=
\sum_{i=1}^m 1_{A_i}(x)\phi_i^{(t)}(x).
\end{equation}
Via the definitions, we have that 
$$
 (G_{t^\Delta}\cdot\CF_\CP)(\bx)
=\sum_{\bj\in\{1,\dots,m\}^{2^d}}
c_\bj
\prod_{\vepsilon\in \vt V_d}f_{\vepsilon 0}(x_\vepsilon)
\prod_{\vepsilon\in V_d}1_{A_{j_\vepsilon}}(x_\vepsilon).
$$
Grouping together all terms of the sum with $j_\bzero=i$ and 
using~\eqref{eq:PH_Product}, we obtain~\eqref{eq:FP_phiit}.
This completes the proof of Theorem~\ref{th:main}.
\qed

\section{Further directions}

We have carried this study of Gowers norms and associated 
dual norms in the setting of 
compact abelian groups.  This leads to a natural question: 
what is the analog of the Inverse Theorem for groups other than 
$\Z_{N}$?   What would be the generalization for other finite groups or 
for infinite groups such as the torus, or perhaps even for totally 
disconnected (compact abelian) groups?

In Section~\ref{sec:embedding}, we 
give examples of functions with small dual norm,
obtained by embedding 
in a nilmanifold.  One can 
ask if this process is general: does one obtain all 
functions with small dual norm, up to a small error in $L^{1}$ 
in this way?  
In particular, for $\Z_{N}$ this would mean that in the Inverse 
Theorem we can replace the family $\CF(d,\delta)$ by 
a family of nilsequences with ``bounded complexity'' 
that are periodic, with period $N$, meaning that 
they all come from embeddings of $\Z_N$ in a nilmanifold.

By the computations in Section~\ref{sec:d=2}, 
we see a difference between $A(2)$ and the dual functions:
the cubic convolution product 
$f$ of  functions belonging to $L^{2}(\mu)$ 
satisfies $\sum|\wh{f}|^{2/3}< \infty$,  while $A(2)$ is the 
family of functions $f$ such that $\sum|\wh{f}(\xi)|< +\infty$. 
It is natural to ask what analogous distinctions are for 
$d> 2$.

\appendix
\section{Proof of the regularity lemma}
\label{appendix:regularity}

We make use of the following version of the Regularity Lemma 
in a Hilbert space introduced in~\cite{Lovasz}:

\begin{lemma}[Lovasz and Szegedy~\cite{Lovasz}]
\label{lemma:lovasz}
Let $K_1, K_2, \ldots$ be arbitrary nonempty subsets of a Hilbert 
space $\CH$.  Then for every $\varepsilon > 0$ and $f\in\CH$, there 
exists 
$k\leq \lceil 1/\varepsilon^2\rceil $ and  $f_i\in K_i$, $i=1, 
\ldots, k$ and $\gamma_1, \ldots, \gamma_k\in \R$ such that for every 
$g\in K_{k+1}$, 
$$
|\langle g, f-(\gamma_1f_1+\ldots + \gamma_kf_k)\rangle| \leq 
\varepsilon\cdot\nnorm g \cdot \nnorm f \ .
$$
\end{lemma}

For the proof of Theorem~\ref{th:regularity}, we follow the 
proof of the strong form of the Regularity Lemma in~\cite{Lovasz}.

\begin{proof}[Proof of Theorem~\ref{th:regularity}]
We only consider the infinite case only, as the proof in the finite 
case is similar.

Choose a sequence  of integers $(1)< s(2)< \ldots$ such that 
$$
(s(1)s(2)\ldots s(i))^{2} < s(i+1)$$ for each $i\in\N$ 
and such that $D/\varepsilon<s(1)$.

Let $\CQ$ be a partition of $Z$ into  at most $s(i)$ sets and 
let $K_i$ consist of $\CQ$-functions.

By Lemma~\ref{lemma:lovasz}, 
there exists $k\leq \lceil 1/\varepsilon ^2\rceil $ and there exists an
$s(1)\ldots s(k)$-step function  $F^*$  such that 
\begin{equation}
\label{eq:bound1}
\Bigl|\int U(F-F^*)\, d\nu\Bigr|\leq  \varepsilon
\end{equation}
for any $s(k+1)$-step function $U$.   Choose $m$ with $D/\epsilon 
<m<s(k+1)$ and refine the partition defining $F^*$ into a partition 
$\CS=\{S_1,\dots,S_m\}$ into $m$ sets. Then 
$F^*$ is a $\CS$-function and the
bound~\eqref{eq:bound1} 
remains valid for every $m$-step function $U$.

Partition each set $S_i$ into
subsets of measure $1/m^2$ and a remainder set of measure 
smaller than $1/m^2$. 
Take the union of all these remainder sets and partition this union into sets 
of measure $1/m^2$. Thus we obtain a partition $\CP=\{A_1,\dots,A_{m^2}\}$
of $Z$ into $m^2$ sets of equal measure.

At least $m^2-m$ of these $m^2$ sets are \emph{good}, 
meaning that the set is included in some set of the partition $\CS$.
Let $G$ denote the union of these good sets and call it the 
\emph{good part} of $Z$.
We have that
$$
 \nu\bigl(Z^D\setminus G^D\bigr)\leq D/m\leq\varepsilon\ .
$$

We claim that if $U$ is an $m$-step function with $|U|\leq 1$, then 
$$
\Bigl\vert \int U(F-F_\CP)\,d\nu
\Bigr\vert \leq 4\varepsilon \ .
$$
To show this, set $U'=\one_{G}\cdot U$. Then 
$$
 \Bigl\vert \int(U-U')(F-F_\CP)\,d\nu\Bigr|
\leq 2\int|U-U'|\,d\nu \leq 2\varepsilon\ .
$$
Moreover,  $U'$ is an $m$-step function with $|U'|\leq 1$ and by 
hypothesis,
$$
\Bigl|  \int U'(F-F^*)\,d\nu\Bigr|\leq\varepsilon
$$
and we are reduced to showing that
$$
 \Bigl\vert\int U'(F^*-F_\CP)\,d\nu
\Bigr\vert \leq \varepsilon \ .
$$

Instead, assume that 
$$
 \int U'(F^*-F_\CP)\,d\nu
 > \varepsilon 
$$
and we derive a contradiction (the opposite bound is proved in the 
same way).

Define a new function $U''$ on $Z^D$.   Set $U''=0=U'$ outside 
$G^D$.  
Let $R$ be a product of good sets. The functions $F^*$ and $F_\CP$ 
are constant on $R$ and thus the function $F^*-F_\CP$ is constant on 
$R$. Define $U''$ on $R$ to be equal to $1$ if this constant is 
positive and to be $-1$ if this constant is negative. Then 
$U''(F^*-F_\CP)\geq U'(F-F_\CP)$ on $R$
and so 
$$
 \int U''(F^*-F_\CP)\,d\nu
 \geq \int U'(F^*-F_\CP)\,d\nu>  \varepsilon \ .
$$
On the other hand, $U''$ is a $\CP$-function and so by definition of 
$F_\CP$,  $\int U''(F-F_\CP)\,d\nu=0$ and 
$$
\int U''(F^*-F)\,d\nu
 >  \varepsilon  \ .
$$
But $U''$ is an $m$-step function with $|U''|\leq 1$ and this integral 
is $<\varepsilon$ by~\eqref{eq:bound1}, leading to a contradiction.

\end{proof}

\end{document}